\documentclass{article}

\usepackage{amsthm,amsmath,amssymb}

\numberwithin{equation}{section}
\theoremstyle{plain}
\newtheorem{thm}{Theorem}[section]
\theoremstyle{remark}

\theoremstyle{plain}
\newtheorem{prop}{Proposition}[section]
\theoremstyle{plain}
\newtheorem{lem}{Lemma}[section]

\begin{document}

\title{Nonparametric density estimation for intentionally corrupted functional data} 

\author{Aurore Delaigle\thanks{School of Mathematics and Statistics and Australian Research Council Centre of Excellence for Mathematical  and Statistical Frontiers, University of Melbourne,  Australia, \textsc{email}: aurored@unimelb.edu.au}  \hspace{5cm} Alexander
Meister\thanks{%
Institut f\"ur Mathematik, Universit\"at Rostock, D-18051 Rostock, Germany,
\textsc{email}: alexander.meister@uni-rostock.de}}

\date{\vspace{-5ex}}

\maketitle

\begin{abstract}
We consider statistical models where functional data are artificially contaminated by independent Wiener processes in order to satisfy privacy constraints. We show that the corrupted observations have a Wiener density which determines the distribution of the original functional random variables, masked near the origin, uniquely, and we construct a nonparametric estimator of that density. We derive an upper bound for its mean integrated squared error which has a polynomial convergence rate, and we establish an asymptotic lower bound on the minimax convergence rates which is close to the rate attained by our estimator. Our estimator requires the choice of a basis and of two smoothing parameters. We propose data-driven ways of choosing them and prove that the asymptotic quality of our estimator is not significantly affected by the empirical parameter selection. We examine the numerical performance of our method via simulated examples.
\end{abstract}

\noindent {\small \textit{Keywords:} classification, convergence rates, differential privacy, infinite-dimensional Gaussian mixtures, Wiener densities. \\[-.3cm]


\section{Introduction} \label{0}

Data privacy is an important feature of a database, where the collected data are transformed and released so as to make it difficult to identify individuals participating in a study. Various privatisation methods are in use, resulting in different privacy constraints of which a popular one is differential privacy. We refer to Wasserman and Zhou~(2010) for a statistical introduction of differential privacy. The privatisation mechanism typically has an impact on the statistical analysis of the data, and one of the research directions in statistical privacy is to find ways of ensuring differential privacy while keeping as much of the information as possible from the original database (see e.g.~Hall et al.,~2013 in the functional data context and Karwa and Slavkovi,~2016 in the setting of synthetic graphs).

One simple way of ensuring differential privacy is to contaminate the data artificially with additive random noise; see for example Wasserman and Zhou~(2010). In the functional data context, Hall et al.~(2013) propose a data release mechanism where the observed functional data are contaminated by adding to each function a random Gaussian process (one per functional observation) that is independent of the original data. Proposition~3.3 in Hall et al.~(2013) roughly says that the data can be made differentially private whenever the scaling noise factor of the Gaussian process is sufficiently large. 

In this paper, we show that if the Gaussian process is a Wiener process and the value of the raw data is masked at the origin, then the contaminated data are differentially private, but they also have a density. This contrasts with the usual functional data setting where the assumption that all measures which are admitted to be the true image measure of functional random variables are dominated by a known basic measure seems very hard to justify. There exists no canonical basic measure such as the Lebesgue measure for finite-dimensional Euclidean data or the Haar measure for data in general locally compact groups.  As a result, inference and descriptive summaries of functional data are often based on pseudo-densities.  See for example Delaigle and Hall~(2010) and  Ciollaro et al.~(2016). Recently, Lin et al.~(2018) considered the estimation of densities for functions which lie in a dense subset $S$ of the Hilbert space $L_2(D)$, where $D$ is a finite interval. There, $S$ is defined as the (non-closed) linear hull of an orthonormal basis of $L_2(D)$ and does not contain the functional data contaminated by Wiener processes that we consider. Privacy issues for functional data are also discussed in the recent work of Mirshani et al. (2017). Therein, the authors also deduce the existence of a Gaussian density for fixed functional observations, but nonparametric estimation of that density is not studied. 

By contrast, with the privatisation process we propose, the privatised functional data have a Radon-Nikodym derivative (thus a true, non pseudo, density) with respect to the Wiener measure. Exploiting the fact that the contaminating distribution is usually known in this context, we consider statistical inference  from such privatised functional data.  

 To our knowledge, most existing nonparametric approaches for estimating  a Wiener density are motivated by diffusion processes. Although these do not include the type of functional data we consider, some of these methods can be applied in our context. See for example Dabo-Niang~(2004a), who suggests  an orthogonal series estimator, Dabo-Niang~(2002, 2004b) and Ferraty and Vieu~(2006), who propose a kernel density estimator (see also Prakasa Rao,~2010a, for a generalisation in the case of diffusion processes), and Prakasa Rao~(2010b) and Chesneau et al.~(2013) who construct a wavelet estimator. See also Ba{\'i}llo et al.~(2011) for a parametric context where the data and the reference measure are Gaussian. However these methods either suffer from slow logarithmic convergence rates, or are derived under abstract assumptions that seem hard to justify in our context, or seem difficult to implement in practice. We propose a fully data-driven estimator which has fast polynomial convergence rates under simple conditions. Although our estimator is motivated by the privacy setting we consider, our results can be extended to more general cases of functional data which have a Wiener density. 

This paper is organised as follows. In Section~\ref{1} we introduce our statistical model and show that the Wiener density exists
 and determines uniquely the image measure of the raw functional random variables masked near zero. Moreover we prove that the privacy constraints are fulfilled when the noise level is sufficiently large. In Section~\ref{2} we construct a nonparametric orthonormal series estimator of the Wiener density and propose  data-driven procedures for choosing the basis (Section~\ref{basis}) and the smoothing parameters (Section~\ref{sec:CV}). In Section~\ref{Theory} we derive an explicit upper bound for the mean integrated squared error of our estimator and show that it achieves polynomial convergence rates under intuitive tail restrictions and metric entropy constraints on the measure of the original data. Functional data problems in which such fast rates are available are rare; usually the achievable rates are only logarithmic or sub-polynomial;  see e.g.~Dabo-Niang~(2004a), Mas~(2012) and Meister~(2016). Finally, we derive a lower bound on the mean integrated square error under our intuitive conditions and we show that choosing the parameters in a data-driven way does not significantly deteriorate the asymptotic performance of our procedure (thus we establish a weak adaptivity result). Numerical simulations are provided in Section~\ref{Sim}. The proofs are deferred to supplement.

\section{Model, data and applications} \label{1}

\subsection{Model and data}\label{sec:model}

We observe functional data $Y_1,\ldots,Y_n$ defined on $[0,1]$, without loss of generality, and which, for reasons such as differential privacy constraints discussed in Section~\ref{0}, have been intentionally contaminated by additive random noise. Specifically, we assume that\begin{equation} \label{eq:model} Y_j = X_j + \sigma W_j\,, \qquad j=1,\ldots,n\,, \end{equation}
where the random functions $X_j$ and $W_j$, $j=1,\ldots,n$, are totally independent. Here, $X_j$ represents the $j$th function of interest, which is corrupted by a standard Wiener process $W_j$ with a deterministic scaling factor $\sigma>0$. Unlike typical measurement error problems where contamination is due to imprecise measurement or unavoidable perturbation, here the data are contaminated artificially and we can assume that $\sigma$ is known. 

We assume that the $X_j$'s take their values in $C_{0,0}([0,1])$ where $C_{0,\ell}([0,1])$ denotes the set of $\ell$ times continuously differentiable (or just continuous when $\ell = 0$) functions $f$ defined on $[0,1]$ that are such that $f(0)=0$. 
The $X_j$'s have an unknown probability measure $P_X$ on the Borel $\sigma$-field $\mathfrak{B}(C_{0,0}([0,1]))$ of $C_{0,0}([0,1])$ where we equip the space $C_{0,0}([0,1])$ with the supremum norm $\|\cdot\|_\infty$.  
Throughout we use the notation $V_j= \sigma W_j$, and we use $V$, $W$, $X$ and $Y$ to denote a generic function that has the same distribution as, respectively, the  $V_j$'s, $W_j$'s, the $X_j$'s and the $Y_j$'s. 
Critically here, the functional data $X_j$ are assumed to satisfy $X_j(0)=0$. Indeed, since $W_j(0)=0$, then $Y_j(0)=X_j(0)$ and if the value of $X_j$ at zero is not masked, then individuals can be identified from $Y_j(0)$. In practice, if the raw data do not satisfy $X_j(0)=0$, they can be pre-masked at zero before the contamination step, for example by replacing $X_j$  by $\widetilde X_j = X_j - X_j(0)$ or $\widetilde X_j=X_j w$ where $w$ is a smooth function such that $w(0)=0$ and $w(1)=1$.  

\subsection{Density of contaminated data and differential privacy}
In this section, we show that the $Y_j$'s have a well defined density with respect to the scaled Wiener measure, and that this density characterises the distribution of the $X_j$'s  uniquely. Finally, we show that the contamination process ensures differential privacy. 

To ensure existence of a density, we need  the following assumption, which we assume throughout this work: \\[2mm]
\noindent {\bf Assumption 1} \\
$X \in C_{0,2}([0,1])$ a.s.. \\[2mm]

\noindent
Under Assumption 1, using Girsanov's theorem (Girsanov, 1960), for any Borel measurable mapping $\varphi$ from $C_{0,0}([0,1])$ to $[0,1]$, we have
\begin{align*}
& E \{\varphi(Y)\} = E \{\varphi(X+\sigma W) \} \\
& \, = E \Big\{\varphi(\sigma W) \exp\Big(\frac1{\sigma} \int_0^1 X'(t) dW(t)\Big) \exp\Big(-\frac1{2\sigma^2} \int_0^1 |X'(t)|^2 dt\Big)\Big\} \\
& \, = E \Big[\varphi(V) E\Big\{\exp\Big(\frac1{\sigma^2} \int_0^1 X'(t) dV(t)\Big) \exp\Big(-\frac1{2\sigma^2} \int_0^1 |X'(t)|^2 dt\Big) \Big| V\Big\}\Big]\,,
\end{align*}
so that, by integration by parts, we have,  a.s.,  
\begin{align} \nonumber
& \frac{dP_Y}{dP_{V}}(V) =  E\Big[ \exp\Big\{\frac1{\sigma^2} \int_0^1 X'(t) dV(t) -\frac1{2\sigma^2} \int_0^1 |X'(t)|^2 dt\Big\} \Big| V\Big] \\  \label{eq:density} & \, =
 \int \exp\Big\{\frac1{\sigma^2} x'(1) V(1) - \frac1{\sigma^2} \int_0^1 x''(t) V(t) dt -\frac1{2\sigma^2} \int_0^1 |x'(t)|^2 dt\Big\} \, dP_X(x)\,. \end{align}

Applying the factorization lemma to this conditional expectation, we deduce that there exists a Borel measurable mapping $f_Y$: $C_{0,0}([0,1]) \to \mathbb{R}$ such that $f_Y(V)$ is equal to the right hand side of \eqref{eq:density} almost surely. This implies that $f_Y$ is the density of $P_Y$ with respect to $P_V$.
Thus the contaminated $Y_j$'s have a density $f_Y$. The next theorem establishes its connection with the measure of the $ X_j$'s.
\begin{thm} \label{P:0}
The functional density $f_Y$ in \eqref{eq:density} characterises the probability measure $P_X$ uniquely.
\end{thm}
\noindent
We deduce from this theorem that inference about $P_X$ (e.g.~goodness-of-fit tests or classification problems; see Section~\ref{sec:applic}) can be performed via  $f_Y$. 
To use this result in practice, it remains to see whether we can estimate $f_Y$ nonparametrically using the data $Y_1,\ldots,Y_n$. This is what we study in Section~\ref{2}.

Throughout we use the notation $\langle\cdot,\cdot\rangle$ for the inner product of $L_2([0,1])$, $\|\cdot\|_2$ for the corresponding norm, and we make the following assumption: \\[2mm]
\noindent {\bf Assumption 2} \\
For some constant $C_{X,1} \in (0,\infty)$, we have that $\|X'\|_2 \, \leq \, C_{X,1} \qquad {\mbox{ a.s.}}\,. $ \\[2mm]
 
The following proposition shows that if the scaling factor $\sigma$ is large enough, the contaminated data are privatised. For the definition of $(\alpha,\beta)$-privacy we refer to Hall et al.~(2013); in our setting this criterion means that
$$ P[x + \sigma W \in B] \leq \exp(\alpha) \cdot P[\tilde{x} + \sigma W \in B] + \beta\,, \quad \forall B \in \mathfrak{B}(C_{0,0}([0,1]))\,, $$
for all $x,\tilde{x} \in C_{0,2}([0,1])$ with $\max\{\|x'\|_2,\|\tilde{x}'\|_2\} \leq C_{X,1}$. 

 \begin{prop} \label{P:privacy}
For any $\alpha,\beta> 0$, choosing $\sigma > 2C_{X,1} \sqrt{2 \log(2/\beta)} / \alpha$ guarantees $(\alpha,\beta)$-privacy of the observation of $Y = X + \sigma W$ under the Assumptions 1 and 2. 
\end{prop}

\subsection{Applications}\label{sec:applic}
The existence of a density for contaminated data has important practical applications. One of them is goodness-of-fit testing. 
Goodness-of-fit tests for functional data have been considered in e.g.~Bugni et al.~(2009). In our context, using the observed  i.i.d.~contaminated functional data $Y_1,\ldots,Y_n$, the problem consists in testing the null hypothesis $H_0 : X_1 \sim P_X$ versus the alternative $H_1 : X_1 \not\sim P_X$ for some fixed probability measure $P_X$ on $\mathfrak{B}(C_{0,0}([0,1]))$. According to Theorem~\ref{P:0}, $H_0$ is equivalent to the claim that $Y_1$ has the  functional density $f_Y = d(P_X*P_V)/dP_V$. Using the estimator $\hat{f}_Y$ of $f_Y$ that we introduce in Section~\ref{2}, a testing procedure could be based on
$$ T(Y_1,\ldots,Y_n) \, := \, \begin{cases} 1\,, & \mbox{ for } \int \big|\hat{f}_Y(y) - f_Y(y)\big|^2 dP_V(y) \, > \, \rho\,, \\
0\,, & \mbox{ otherwise,} \end{cases} $$
where $\rho$ is a threshold parameter. In Theorem~\ref{P:1} we derive an upper bound on the mean integrated squared error of our estimator $\hat{f}_Y$.  Using the Markov inequality, we deduce that the test can attain any given significance level $\alpha>0$ if we select $\rho$ larger or equal to the ratio of this upper bound and $\alpha$. While this gives some insights about $\rho$, this upper bound does not provide a data-driven rule for selecting $\rho$ in practice. The latter is a difficult problem; for example, it requires deriving the asymptotic distribution of the fully data-driven estimator. This goes beyond the scope of this paper and we leave this issue open for future research.

Another interesting application is classification, which, in our context, can be expressed as follows. We observe training contaminated data pairs $(Y_{i},I_{i})$, $i=1,\ldots,n$, where $Y_{i}=X_{i}+V_{i}$, the $X_i$'s come from  two distinct populations $\Pi_0$ and $\Pi_1$ and the class label  $I_i=k$ if $X_i$ comes from population $\Pi_k$, for $k=0,1$.   The  $V_i$'s are Wiener processes independent of the $X_i$'s, and are identically distributed within each population, but  the scaling noise parameter $\sigma$ need not be the same for the two populations. Using these data, the goal is to classify in $\Pi_0$ or $\Pi_1$ a new random curve $Y = X + V$, where $X$ comes from either $\Pi_0$ or $\Pi_1$, but whose class label is unknown. 

It is well known in general classification problems that  the optimal classifier is the Bayes classifier which, adapted to our context,  assigns a curve to $\Pi_1$ if $E(I|Y=y)>1/2$ and to $\Pi_0$ otherwise. 
In the case where the probability measures $P_{Y,0}$  and $P_{Y,1}$ of the $Y_i$'s that originate from, respectively, $\Pi_0$ and $\Pi_1$, have well defined densities  $f_{Y,0}$ and $f_{Y,1}$, the Bayes classifier can be expressed as: assign $Y$ to $\Pi_1$ if $\pi_1 f_{Y,1}(Y)>\pi_0 f_{Y,0}(Y)$, and to $\Pi_0$ otherwise, where $\pi_k=P(I=k)$.  In the particular Gaussian case,   Ba{\'i}llo et al.~(2011) showed that these densities are well defined and showed how to estimate them.

 In our case the $Y_i$'s are generally not Gaussian but they have functional densities $f_{Y,k} = dP_{Y,k} / dP_V$, for $k=0,1$. Since $P_{X,0}\neq P_{X,1}$ implies that   $f_{Y,0}\neq f_{Y,1}$ (see Theorem~\ref{P:0}), these densities can be used for classification of $X$ from observations on $Y$ in the optimal Bayes classifier. There, in practice we classify  $Y$ in $\Pi_1$ if $\pi_1 \hat f_{Y,1}(Y) \geq \pi_0 \hat f_{Y,0}(Y)$ and in $\Pi_0$ otherwise, where for $k=0,1$,  $\hat f_{Y,k}$ denotes the estimator of $f_{Y,k}$ from Section~\ref{2} constructed from the training data $Y_i$ for which $I_i=k$.

There exist many other classification procedures for functional data, often based on pseudo-densities or finite dimensional approximations. However, Delaigle and Hall~(2012) pointed that, except in the Gaussian case, often such projections do not ensure good finite sample  performance. See for example Hall et al.~(2001), Ferraty and Vieu~(2006), Escabias et al.~(2007),  Preda et al.~(2007) and Shin~(2008). See also Dai et al.~(2017) for a recent example, where the authors approximate the densities in the two populations by the finite dimensional surrogate densities proposed in Delaigle and Hall~(2010); see Delaigle and Hall~(2013) for a related classifier.

\section{Methodology} \label{2}
In this section we consider the problem of estimating the functional density $f_Y$ nonparametrically.

\subsection{Existing methods}\label{sec:existing}
Nonparametric estimation of a density for stochastic processes whose probability measure has a Radon-Nikodym derivative with respect to the Wiener measure has been considered by several authors. 
In Dabo-Niang~(2002, 2004b), the author proposes to use a kernel density estimator; see also Prakasa Rao~(2010a). This estimator is simple but it suffers from slow logarithmic convergence rates, which are reflected by its practical performance.
A wavelet estimator with polynomial convergence rates was proposed by Prakasa Rao~(2010b) and Chesneau et al.~(2013), but their conditions are quite technical 
and it is not clear how their parameters can be chosen in practice. Moreover, their theory is derived under abstract high level conditions which might not be easily satisfied in our context.

A simpler estimator is the orthogonal series estimator of Dabo-Niang~(2004a), defined as follows. Let $\{{\varphi}_j\}_{j\in \mathbb{N}}$ denote an orthonormal basis of real-valued functions of $[0,1]$, where each $\varphi_j\in L_2([0,1])$, and let $(H_j)_{j\geq 1}$ denote the scaled Hermite polynomials  defined by $ H_k(x)  =(-1)^k  \phi^{(k)}(x) /\{\phi(x) \sqrt{k!}\}$, for all integer $k\geq 0$, where $\phi(x)=\exp(-x^2/2)/\sqrt{2\pi}$. Also, for $x\in C_0([0,1])$, let 
\begin{equation}
\beta'_{x,\ell}=\int_0^1 \varphi_\ell(t)dx(t)\,.
\label{eq:eps}
\end{equation}
Using results from Cameron and Martin~(1947), the author notes that, as $K\to\infty$, the Fourier-Hermite series $(\Psi_{k_1,\ldots,k_K})_{0\leq k_1\leq K,\ldots,0\leq k_K\leq K}$, where, for $x\in C_0([0,1])$,
\begin{equation}
\Psi_{k_1,\ldots,k_K}(x)\equiv H_{k_1,\ldots,k_K}( \beta'_{x,1},\ldots, \beta'_{x,K})\equiv \prod_{\ell=1}^K H_{k_\ell} (\beta'_{x,\ell})\,, \label{FourierHermite}
\end{equation}
forms an orthonormal basis of the Hilbert space of all square-integrable $C_0([0,1])$-valued random variables with respect to the Wiener measure.
Motivated by this, the author proposes to estimate the Wiener density $f_T$ of functional data $T_1,\ldots,T_n$ (that have a Wiener density) by 
\begin{equation}
\hat{f}_T^K(x)  \, = \, \sum_{k_1,\ldots,k_K=0}^K  \, \frac1n \sum_{j=1}^n H_{k_1,\ldots,k_K}( \beta'_{T_j,1},\ldots, \beta'_{T_j,K})  \cdot H_{k_1,\ldots,k_K}( \beta'_{x,1},\ldots, \beta'_{x,K})\,,\label{EstDabo}
\end{equation}
where $K$ is a smoothing parameter. This estimator is very attractive for its simplicity. 
However, a drawback is that the rates derived by Dabo-Niang~(2004a) are logarithmic. In the next two sections, using a two-stage approximation approach (first a sieve approximation of $f_Y$ and then an estimator of the approximation), we are able to  introduce a different regularisation scheme which involves two parameters. This increases the flexibility of the estimator, which, as we shall see, enables us to obtain polynomial convergence rates. Moreover we provide data-driven choices of the basis and the threshold parameters.

\subsection{Finite-dimensional approximation of $f_Y$} \label{fin}
Recall from \eqref{eq:density} that for $V=\sigma W$ with $W$ a standard Wiener process, we have
$$f_Y(V)=E\Big[ \exp\Big\{\frac1{\sigma^2} \int_0^1 X'(t) dV(t) -\frac1{2\sigma^2} \int_0^1 |X'(t)|^2 dt\Big\} \Big| V\Big]\,, \qquad\mbox{a.s.}\,,$$
and that our goal is to estimate $f_Y$ from data $Y_1,\ldots,Y_n$. 
Instead of directly expressing $f_Y$ in the Fourier-Hermite basis at \eqref{FourierHermite}, we first construct a sieve approximation of $f_Y$, and then (see Section~\ref{2.est}) we express our sieve approximation in the Fourier-Hermite basis.

Using the notation $\beta'_{x,\ell}=\int_0^1 \varphi_\ell(t)dx(t)$ from Equation \eqref{eq:eps}, where $\{{\varphi}_j\}_{j\in \mathbb{N}}$ is a real-valued orthonormal basis of $L_2([0,1])$, we can write
\begin{equation} 
 \int_0^1 X'(t) \,dV(t) -\frac1{2} \int_0^1 |X'(t)|^2 \,dt
=\sum_{j=1}^\infty \beta_{X,j}' \cdot \beta_{V,j}' -\frac1{2} \sum_{j=1}^\infty {\beta_{X,j}'}^2 
\,, \label{expandXV}
\end{equation}
where the infinite sums should be understood as mean squared limits. Truncating the sums to $m$ terms, with $m\geq 1$ an integer, this suggests that we can approximate $f_Y(V)$ by $f_Y^{[m]}(\beta_{V,1}',\ldots,\beta_{V,m}') $,   where, for all $s_1,\ldots,s_m\in \mathbb{R}$,
\begin{align} 
f_Y^{[m]}(s_1,&\ldots,s_m) \, 
=  E \Big\{\exp\Big(\frac1{\sigma^2} \sum_{j=1}^m \beta'_{X,j} \cdot s_j - \frac1{2\sigma^2} \sum_{j=1}^m {\beta'_{X,j}}^2 \Big)\Big\}\notag\\
= &\, \exp\Big(\frac1{2\sigma^2} \sum_{j=1}^m s_j^2 \Big) \int \exp\Big\{-\frac1{2\sigma^2} \sum_{j=1}^m \big(s_j - x_{j}\big)^2\Big\} dP_{X,m}(x_1,\ldots,x_m)\,,\label{fYm}
\end{align}
and where $P_{X,m}$ denotes the measure of $(\beta'_{X,1},\ldots,\beta'_{X,m})$.

The following lemma shows that, as long as $m$ is sufficiently large, \linebreak$f_Y^{[m]}(\beta'_{V,1},\ldots,\beta'_{V,m})$ is a good approximation to $f_Y(V)$, where $V$ denotes a generic $V_i\sim P_V$.
\begin{lem} \label{L:1}
Let $\mathfrak{A}_m$ denote the $\sigma$-field generated by $\beta'_{V_1,1},\ldots,\beta'_{V_1,m}$. Under Assumptions 1 and 2, \\ 
(a)\;  $f_Y^{[m]}(\beta'_{V_1,1},\ldots,\beta'_{V_1,m})=E\{f_Y(V_1) |\mathfrak{A}_{m}\}$ a.s.\,, \\
(b)\;  we have
\begin{align*} E \big|f_Y^{[m]}(\beta'_{V_1,1},\ldots,\beta'_{V_1,m}) & - f_Y(V_1)\big|^2 \\ & \, \leq \,  \frac1{\sigma^2} \cdot \exp\big(C_{X,1}^2 / \sigma^2\big) \cdot \Big(\sum_{j,j'>m} \big|\langle \varphi_j , \Gamma_X \varphi_{j'} \rangle\big|^2\Big)^{1/2}\,, 
\end{align*}
where  the linear operator $\Gamma_X : L_2([0,1]) \to L_2([0,1])$ is defined by
\begin{equation} \label{eq:GammaX}
\big(\Gamma_X f\big)(t) \, = \, E \Big\{X'(t) \int_0^1 X'(s) f(s)\, ds\Big\}\,, \qquad t\in [0,1],\, f\in L_2([0,1])\,. \end{equation}
\end{lem}   

\noindent Since $\Gamma_X$ is a self-adjoint and positive-semidefinite Hilbert-Schmidt operator, the upper bound in Lemma~\ref{L:1}(b) is finite for any orthonormal basis $\{\varphi_j\}_j$ of $L_2([0,1])$, and converges to zero as $m\to\infty$. Indeed, Assumption 2 guarantees that
$  \sum_{j,j'} \big|\langle \varphi_j , \Gamma_X \varphi_{j'} \rangle\big|^2 \leq  E \|X_1'\|_2^4 \leq  C_{X,1}^4  <  \infty$. 
If $X$ (and hence $X'$) is centered then $\Gamma_X$ coincides with the covariance operator of $X'$.

\subsection{Estimating the sieve approximation of $f_Y$} \label{2.est}
Next we show how to estimate $f_Y^{[m]}$  using a Fourier-Hermite series.  
For this, let  $P_{Y,m}$ and $f_{Y,m}$ denote, respectively, the measure and the $m$-dimensional Lebesgue density of the observed random vector $(\beta'_{Y_j,1},\ldots,\beta'_{Y_j,m})$, where
\begin{align*}
\beta'_{Y_j,k} \, = \, \int_0^1 \varphi_k(t) \,dY_j(t) \, = \, \beta'_{X_j,k}  + \beta'_{V_j,k}\,, \quad j=1,\ldots,n;\, k=1,\ldots,m\,.
\end{align*}
Let $g_\sigma$ denote the $N(0,\sigma^2 I_m)$-density, with $I_m$ the $m\times m$-identity matrix, let $L_{2,g_\sigma}(\mathbb{R}^m)$ denote the Hilbert space of  Borel measurable functions $f : \mathbb{R}^m \to \mathbb{R}$ which satisfy 
$ \|f\|_{g_\sigma}^2 \, \equiv \, \int |f(t)|^2 g_\sigma(t) dt \, < \, \infty$, and let 
$\langle\cdot,\cdot\rangle_{g_\sigma}$ denote the inner product of $L_{2,g_\sigma}(\mathbb{R}^m)$.

It is easy to deduce from \eqref{fYm} that
\begin{equation} 
f_Y^{[m]}(s_1,\ldots,s_m) \, = \, f_{Y,m}(s_1,\ldots,s_m)/g_\sigma(s_1,\ldots,s_m)\,,\label{eq:proj} 
\end{equation}
and it can be proved that $f_Y^{[m]}\in L_{2,g_\sigma}(\mathbb{R}^m)$. Therefore, if $\Psi_1,\Psi_2,\ldots$ is an orthonormal basis of $L_{2,g_\sigma}( \mathbb{R}^m)$, we can write
\begin{align*}
& f_Y^{[m]} 
=\sum_{k=1}^\infty \alpha_k \, \Psi_k, \\ &  \alpha_k= \langle \Psi_k,f_Y^{[m]}  \rangle_{g_\sigma}=\int \Psi_k(y)f_{Y,m}(y)\,dy= E\{\Psi_k(\beta'_{Y,1},\ldots,\beta'_{Y,m})\}  \,.
\end{align*} 
Now the sequence $(H_{k_1,\ldots,k_m})_{k_1,\ldots,k_m \geq 0}$ of functions
$
H_{k_1,\ldots,k_m}(x_1,\allowbreak\ldots,x_m) =  \prod_{j=1}^m\allowbreak H_{k_j}(x_j)$ defined at \eqref{FourierHermite} 
 forms an orthonormal basis of $L_{2,g_1}(\mathbb{R}^m)$. Thus we can take $\Psi_k(\cdot)= H_{k_1,\ldots,k_m}(\cdot /\sigma)$. 
To estimate $f_Y^{[m]}$, we replace $\alpha_k$ by $\hat\alpha_k=n^{-1}\sum_{j=1}^n \Psi_k(\beta'_{Y_j,1},\ldots,\beta'_{Y_j,m})$. 

Finally, for $U$ a functional random variable  independent of $Y_1,\ldots,Y_n$ which has a density with respect to $P_V$, we define our estimator of $f_Y(U)$ by 
\begin{align} \nonumber 
& \hat{f}_Y^{[m,K]}(U)  \\ \nonumber & \, = \hspace*{-4mm} \sum_{k_1,\ldots,k_m\geq 0} \frac1n \sum_{j=1}^n H_{k_1,\ldots,k_m}(\beta'_{Y_j,1}/\sigma,\ldots,\beta'_{Y_j,m}/\sigma) H_{k_1,\ldots,k_m}\big(\beta'_{U,1}/\sigma,\ldots,\beta'_{U,m}/\sigma\big) \\ & \hspace{4cm}\times \omega_K(k_1+\cdots+k_m) \,1\{k_1+\cdots+k_m\leq K\}\,,\label{eq:est}
\end{align}
where   $K\geq 0$ is a truncation parameter and  $0\leq \omega_K(x)\leq 1$ a continuous function defined on $[0,K]$. The term $\omega_K(k_1+\cdots+k_m) \,1\{k_1+\cdots+k_m\leq K\}$ prevents the $k_i$'s from being too large, which controls the variability of the estimator. Using wavelet terminology, the function $\omega_K$ dictates whether the $k_i$'s are chosen by a soft or a hard rule. Specifically, a hard rule corresponds to $\omega_K \equiv 1$: here all $k_i$'s summing to at most $K$ are given equal weight and as $K$ increases, new indices appear and play as big a role as older ones. For a soft rule, $\omega_K(x)$ is taken to be a smooth decreasing function of $x$, e.g.~$\omega_K(x)=1-x/(K+1)$; as $K$ increases, new indices start playing a role but have less weight than former ones.

A major difference between (\ref{eq:est}) and Dabo-Niang's~(2004a) estimator at \eqref{EstDabo} is our regularisation scheme: because of the two-step construction of our estimator (sieve approximation followed by basis expansion), we do not use all the indices $(k_1,\ldots,k_K) \in \{0,\ldots,K\}^K$. Instead we use  $(k_1,\ldots,k_m) \in  \{0,\ldots,K\}^m$ such that $k_1+\ldots+k_m \leq K$, and we assign a weight $\omega_K(k_1+\ldots+k_m)$ to each group of $m$ indices.  As we will see in the next sections, our use of a second parameter $m$ and the restriction we put on $k_1+\ldots+k_m$ drastically improve the quality of the estimator, both theoretically and practically. Moreover, in Section~\ref{basis}, we introduce a data-driven way of choosing the basis $\{{\varphi}_j\}_{j\in \mathbb{N}}$ used to construct the coefficients $\beta'_{Y_j,k}$ and $\beta'_{U,k}$.

\subsection{Choosing the  $\varphi_j$'s} \label{basis}
To compute our estimator in practice, we need to choose the basis $\{\varphi_j\}_j$ used in \eqref{eq:eps}. Lemma~\ref{L:1}(b) implies that if we take the $\varphi_j$'s equal to the eigenfunctions of $\Gamma_X$, ordered such that the sequence of corresponding eigenvalues $(\lambda_j)_j$ decreases monotonically, then
\begin{equation} \label{eq:PC}
 E \big|f_Y^{[m]}(\beta'_{V_1,1},\ldots,\beta'_{V_1,m}) - f_Y(V_1)\big|^2 \, \leq \,  \frac1{\sigma^2}\cdot \exp\big(C_{X,1}^2 / \sigma^2\big)\cdot \Big(\sum_{j>m} \lambda_j^2\Big)^{1/2}\,.
\end{equation} 
This bound  decreases monotonically as $m$ increases, which gives an indication that the first $m$ terms of the basis capture some of the main characteristics of $f_Y$.

Of course, in practice $\Gamma_X$ is unknown and thus the $\varphi_j$'s are unknown. Thus we need to estimate $\Gamma_X$, but a priori this does not seem to be an easy task because, up to some mean terms, $\Gamma_X$ is the covariance function of the first derivative $X'$ of $X$. If we could observe $X'_1,\ldots,X'_n$, we could use standard covariance estimation techniques such as those in  Hall and Hosseini-Nasab~(2006), Mas and Ruymgaard~(2015) and Jirak~(2016). However we only observe the contaminated $Y_j$'s. If the $Y_j$'s were differentiable, we could take their derivative and estimate  $\Gamma_X$ and its eigenfunctions as in the references just cited. However they are not differentiable and  we cannot take such a simple approach.

Instead, we propose the following approximation procedure. Let $\{\psi_j\}_j$ denote an orthonormal basis of $L_2([0,1])$, and recall that $\varphi_\ell$ denotes the eigenfunction of $\Gamma_X$ with eigenvalue $\lambda_\ell$, where $\lambda_1\geq \lambda_2 \geq\cdots$. 
In the supplement we show that, for all $k\geq 1$, 
\begin{equation} \label{eq:PC.new.1} \sum_{j=1}^\infty \varphi_{\ell,j} \, \langle \psi_k, \Gamma_X \psi_j\rangle\,  \, = \, \lambda_\ell\, \varphi_{\ell,k}\,, \end{equation}
where $\varphi_{\ell,j} = \langle \varphi_\ell , \psi_j \rangle$. If we take the $\psi_j$'s to be continuously differentiable and such that $\psi_j(0)=\psi_j(1)=0$, for example if $\{\psi_j\}_j$ is the Fourier sine basis, then for $j,k=1,2,\ldots$, we have
\begin{equation}
\langle \psi_k, \Gamma_X \psi_j\rangle =   {\cal M}_{j,k}  -  \sigma^2 \cdot 1\{j=k\}\,, \label{eq:rem.eq1}
\end{equation}
where $ {\cal M}_{j,k} =  \int_0^1 \psi_j'(t) \int_0^1 E  \big\{Y(t) Y(s)\big\} \psi_k'(s)\,ds\, dt$ (see the proof in the supplement). We propose to approximate   $\varphi_\ell$  by $\sum_{j=1}^M \hat\varphi_{\ell,j}  \psi_j$, with $M$ a large positive integer, where $\hat\varphi_{\ell,j}$ denotes an estimator of $\varphi_{\ell,j}$. Next we show how to compute $\hat\varphi_{\ell,1},\ldots,\hat\varphi_{\ell,M}$ from our data. First, combining \eqref{eq:PC.new.1}  and \eqref{eq:rem.eq1}, we have $\sum_{j=1}^\infty \varphi_{\ell,j} \, \big({\cal M}_{j,k}  -  \sigma^2 \cdot 1\{j=k\}\big)=\lambda_\ell\, \varphi_{\ell,k}$ so that
\begin{equation}
 \sum_{j=1}^M \varphi_{\ell,j} \, \big({\cal M}_{j,k}  -  \sigma^2 \cdot 1\{j=k\}\big) = \lambda_\ell\, \varphi_{\ell,k}+R_{k,\ell}\,,\label{lambdafk}
\end{equation}
where $R_{k,\ell}$ is a remainder term resulting from the truncation of the sum to $M$ terms.  Let $I_M$ and ${\cal M}$ denote, respectively, the $M\times M$-identity matrix and the $M\times M$-matrix whose  components are defined by $ {\cal M}_{j,k}$, $j,k=1,\ldots,M$, and let $\varPhi_\ell = (\varphi_{\ell,1},\ldots,\varphi_{\ell,M})^T$ and $R_\ell=(R_{1,\ell},\ldots,R_{M,\ell})^T$. Then \eqref{lambdafk} implies that 
$({\cal M} - \sigma^2 I_M)\, \varPhi_\ell \, = \, \lambda\, \varPhi_\ell +R_\ell$. 

Note that $|R_\ell|$ shrinks to zero as $M\to\infty$ since
$ |R_\ell|^2 \leq C_{X,1}^4 \sum_{j>M} |\varphi_{\ell,j}|^2$.
Thus, $\big({\cal M} - \sigma^2 I_M) \varPhi_\ell\approx \lambda_\ell \varPhi_\ell$, which motivates us to approximate $\varPhi_\ell$ by  the unit eigenvector $v_\ell$ of the matrix ${\cal M} - \sigma^2 I_M$ correponding to the $\ell$th largest eigenvalue. Now,  $\big({\cal M} - \sigma^2 I_M) v_\ell = \lambda_\ell v_\ell$ implies that ${\cal M}v_\ell= (\lambda_\ell  + \sigma^2)v_\ell$. Thus, $v_\ell$ is also the eigenvector of ${\cal M}$ corresponding to its $\ell$th largest eigenvalue. Of course, ${\cal M}$ is unknown but it can be estimated by
\begin{equation}
 \label{eq:Mhat} 
 \hat{{\cal M}} =  \frac1{n} \sum_{\ell=1}^n \Big\{\int_0^1 \int_0^1 \psi_j'(t) Y_\ell(t) Y_\ell(s) \psi_k'(s) ds\, dt\Big\}_{j,k=1,\ldots,M}\,. 
\end{equation}

For $\ell=1,\ldots,M$, let $\hat{v}_\ell$ denote the $M$ unit eigenvectors of $\hat{{\cal M}}$ (ordered so that the corresponding eigenvalues decrease monotonically).  We propose to estimate $\varPhi_\ell$ by  $\hat{\varPhi}_\ell= (\hat\varphi_{\ell,1},\ldots,\hat\varphi_{\ell,M})^T=\hat v_\ell$. Finally, we estimate $\varphi_\ell$ by
$
\hat{\varphi}_\ell = \sum_{j=1}^M  \hat\varphi_{\ell,j} \, \psi_j 
$.

\subsection{Choosing the parameters $M$, $m$ and $K$}\label{sec:CV}

To compute  the estimator at \eqref{eq:est} in practice, we need to choose three parameters:  
$M$, the parameter used in Section~\ref{basis} to construct the basis functions $\varphi_j$ employed to compute the projections in \eqref{eq:eps},
$m$, a parameter which dictates the dimension of our approximation of $f_Y$ by $f_Y^{[m]}$ at \eqref{fYm}, 
and $K$, the truncation parameter of our orthogonal series expansion at \eqref{eq:est}. 
Having the $\hat\varphi_j$'s close to the eigenfunctions of $\Gamma_X$ is likely to give better practical performance, but it is not necessary for the consistency of our estimator. This suggests that the choice of $M$ is not crucial and we take  $M=20$. 
By contrast, $m$ and $K$ are important smoothing parameters which influence consistency and need to be chosen with care.  
We suggest choosing $(m,K)$ by minimising the cross-validation (CV)  criterion  
\begin{equation} \label{eq:CV} \mbox{CV}(m,K) = \int \big|\hat{f}_Y(v)\big|^2 dP_V(v) - \frac{2}{n} \sum_{i=1}^n \hat{f}_Y^{(-i)}(Y_i)\,, \end{equation}
with $\hat{f}_Y^{(-i)}$ defined in the same way as the estimator at \eqref{eq:est}, except that it uses only the data $Y_1,\ldots,Y_{i-1},Y_{i+1},\ldots,Y_n$. 
To compute the integral at \eqref{eq:CV} we generate a large sample (we took a sample of size $10000$ in our numerical work) of $V_j$'s from $P_V$ and approximate the integral by the mean of the $|\hat{f}_Y(V_j)|^2$'s. 

As in standard nonparametric density estimation problems, our  cross-validation criterion can have multiple local minima and the global minimum is not necessarily a good choice. In case of multiple local minima, we choose the one that produces the smallest value of $m+K$. Moreover, when minimising $CV(K,m)$ we discard all pairs of values $(K,m)$ for which  more than 50\% of the  $\hat{f}_Y^{(-i)}$'s or of the $\hat{f}_Y$'s are negative. For the non discarded $(K,m)$'s, we replace each negative $\hat{f}_Y^{(-i)}(Y_i)$ and $\hat{f}_Y(V_j)$ by recomputing those estimators by repeatedly replacing $K$ by $K-1$ and $m$ and $m-1$ until the negative estimators become positive.

\section{Theoretical properties} \label{Theory}
In this section we derive theoretical properties of our estimator. For simplicity we derive our results in the case where the weight function $\omega_K$ in \eqref{eq:est} is equal to $1$. Similar results can be established for a more general weight function, but at the expense of even more technical proofs. In Section~\ref{sec:theory:finite} we derive an upper bound on the mean integrated squared error of our estimator which is valid for all $n$. Next, in Section~\ref{ConvRates} we derive asymptotic properties of our estimator.

\subsection{Finite sample properties}\label{sec:theory:finite}
In the next theorem, we give an upper bound on the mean integrated squared error 
$$ {\cal R}(\hat{f}_Y^{[m,K]},f_Y) =  E \int \big|\hat{f}_Y^{[m,K]}(v) - f_Y(v)\big|^2 dP_V(v)$$
of the estimator at  \eqref{eq:est} in the case where the orthonormal basis $\{\varphi_j\}_j$ and the parameters $m$ and $K$ are deterministic. Our result is non asymptotic and is valid for all $n$.

\begin{thm} \label{P:1}
Under Assumptions 1 and 2 and the selection $\omega_K\equiv 1$, we have
$ {\cal R}(\hat{f}_Y^{[m,K]},f_Y) \, \leq \, {\cal V} + {\cal B} + {\cal D}$, 
where
\begin{align*}
{\cal V} & =  \frac1n \, \exp\big(K C_{X,1}^2/\sigma^2\big) \cdot {K+m \choose K}\,, \ \ 
{\cal B}  =  \inf_{h \in {\cal H}_{m,K}} \big\|f_Y^{[m]}(\sigma\cdot) - h\big\|_{g_1}^2\,, \\
{\cal D} & =  \frac1{\sigma^2} \cdot \exp\big(C_{X,1}^2 / \sigma^2\big) \cdot \Big(\sum_{j,j'>m}  \big|\langle \varphi_j, \Gamma_X \varphi_{j'} \rangle\big|^2\Big)^{1/2} \,,
\end{align*}
and where  ${\cal H}_{m,K}$ denotes the linear hull of the $H_{k_1,\ldots,k_m}$'s for which $k_1+\cdots+k_m\leq K$.
\end{thm}
In Theorem~\ref{P:1}, ${\cal V}$ represents a  variance term while ${\cal B}$ represents a bias term which depends on smoothness properties of $f_Y^{[m]}$. Both are typical of nonparametric estimators, but the term ${\cal D}$ is of a different type. It reflects the error of the finite-dimensional approximation of the density $f_Y$ by the function $f_Y^{[m]}$.

\subsection{Asymptotic properties} \label{ConvRates}
Next we derive asymptotic properties of our density estimator. For this, we need an additional assumption which will be used when dealing with  the term ${\cal D}$ from Theorem~\ref{P:1}: \\[2mm]

\noindent {\bf Assumption 3} \\
There exist constants $C_{X,2}, C_{X,3} \in (0,\infty)$ and  $\gamma>0$ such that
$$ \sum_{j,j'>m} \Big|\int_0^1 \varphi_j(s) \big(\Gamma_X \varphi_{j'}\big)(s) ds\Big|^2 \, \leq \, C_{X,2} \cdot \exp\big(- C_{X,3} m^{\gamma}\big)\,, \quad \forall m\in \mathbb{N}\,. $$ \\[2mm]

For example, if $X_1$ is centered and $\{\varphi_j\}_j$ is the principal component basis with eigenvalues $\lambda_1\geq \lambda_2\geq\cdots$ discussed in Section~\ref{basis}, then Assumption 3 is satisfied as soon as
$\sum_{j=1}^\infty \exp(C_{X,3}' j^\gamma)\cdot \lambda_j^2\, < \, \infty$ 
for some $C_{X,3}'>C_{X,3}$. In this case, Assumption 3 can be interpreted as an exponential decay of the eigenvalues of $\Gamma_X$; concretely Assumption 3 is satisfied if there exist some $C_{X,3}''>C_{X,3}'>C_{X,3}$ and some $C_{X,3}'''>0$ such that $\lambda_j \leq C_{X,3}''' \, \exp(- C_{X,3}'' j^\gamma / 2)$ for all integer $j\geq 1$. 

The next theorem establishes an upper bound on the convergence rates of the mean integrated squared error of our estimator $\hat{f}_Y^{[m,K]}$ as the sample size $n$ tends to infinity. We establish the upper bound uniformly over the class  ${\cal F}_X = {\cal F}_X\big(C_{X,1},C_{X,2},C_{X,3},\gamma,\{\varphi_j\}_j\big)$  of all admitted image measures of $X_1$ such that Assumptions 1 to 3 are satisfied for some deterministic orthonormal basis $\{\varphi_j\}_j$ of $L_2([0,1])$. The next three theorems consider functions in this class, which implies that they are derived under Assumptions 1 to 3.

\begin{thm} \label{T:1}
Assume that $\gamma \in (0,1)$ and select the weight function $\omega_K\equiv 1$ and the parameters $K$ and $m$ such that 
$ K = K_n = \lfloor \gamma (\log n) / \log(\log n) \rfloor$, $m  = m_n =  \lfloor \big(C_M \cdot \log n\big)^{1/\gamma}\rfloor$, 
for some finite constant $C_M > 2/C_{X,3}$. Then our estimator $\hat{f}^{[m,K]}$ satisfies
$$ \limsup_{n\to\infty} \sup_{P_X \in {\cal F}_X} \log \big\{ {\cal R}\big(\hat{f}_Y^{[m,K]},f_Y\big)\big\} / \log n \, \leq \, -\gamma\,. $$
\end{thm}

Theorem~\ref{T:1} shows that the risk of our estimator converges to zero faster than  ${\cal O}(n^{-\gamma'})$ for any $\gamma'<\gamma < 1$. In particular, our estimator achieves polynomial convergence rates, which is usually impossible in problems of nonparametric functional regression or density estimation. In standard problems of that type where the data range over an infinite-dimensional space, only logarithmic or sub-algebraic rates can usually be achieved   (see e.g.~Mas,~2012, Chagny and Roche,~2014 and Meister,~2016). In our case the dimension of the data is infinite as well; however the density $f_Y$ forms an infinite-dimensional Gaussian mixture and its smoothness degree is sufficiently high to overcome the difficulty caused by high dimensionality.

The next  theorem provides an asymptotic lower bound  for the problem of estimating $f_Y$ nonparametrically. For simplicity we restrict to the case where $C_{X,1}=1$.

\begin{thm} \label{T:2}
Assume that $\gamma \in (0,1)$ and let $C_{X,1} = 1$ in Assumption 2. 
Moreover, assume that the orthonormal basis $\{\varphi_j\}_j$ of $L_2([0,1])$ is such that all $\varphi_j$'s are continuously differentiable. Then, for any sequence  $(\hat{f}_n)_n$ of estimators of $f_Y$ computed from the data $Y_1,\ldots,Y_n$, we have
$$ \liminf_{n\to\infty} \sup_{P_X \in {\cal F}_X} \log \big\{ {\cal R}\big(\hat{f}_n,f_Y\big)\big\} / \log n \, \geq \, -\gamma + (\gamma-1)^2 / (\gamma-2)\,. $$
\end{thm}

We learn from the theorem that, in this problem, no nonparametric estimator can reach the  parametric squared  convergence rate $n^{-1}$. This is significantly different from the simpler problem of nonparametric estimation of one-dimensional Gaussian mixtures, where the parametric rates are achievable up to a logarithmic factor (see Kim,~2014). Note that the upper bound in Theorem~\ref{T:1} is usually larger than  the lower bound in Theorem~\ref{T:2}, although the two bounds are very close to each other for $\gamma$ close to $1$. 
Rather than our estimator being suboptimal, we suspect that our lower bound is not sharp enough. Deriving the exact minimax rates seems a very challenging open problem for future research. 

As is standard in nonparametric estimation problems requiring the choice of smoothing parameters, Theorem~\ref{T:1} was derived under a deterministic choice of $m$ and $K$. Next we establish an asymptotic result in the case where  $(\hat{m},\hat{K})$ is chosen by cross-validation as at  (\ref{eq:CV}), where minimisation is performed over the mesh
\begin{equation} \label{eq:T.3.G} G \, = \, \big\{\lfloor \log n \rfloor,\ldots,\lfloor (\log n)^{1/\gamma_0}\rfloor\big\} \times \big\{1,\ldots,\lfloor (\log n) / \log(\log n) \rfloor\big\}\,, \end{equation}
for some constant $\gamma_0 \in (0,\gamma)$.  The following theorem shows that the convergence rates from Theorem~\ref{T:1} can be maintained at least in a weak sense. 
\begin{thm} \label{T:3}
Our estimator $\hat{f}_Y^{[\hat{m},\hat{K}]}$, where $\omega_K\equiv 1$ and $(\hat{m},\hat{K})$ is selected by cross-validation over the mesh $G$ at \eqref{eq:T.3.G}, satisfies
$$ \lim_{n\to\infty} \sup_{P_X \in {\cal F}_X} P\Big\{ n^{\gamma} \int \big|\hat{f}_Y^{[\hat{m},\hat{K}]}(x) - f_Y(x)\big|^2 dP_V(x) \geq n^{d}\Big\} \, = \, 0\,, $$
for all $\gamma \in [\gamma_0,1)$ and $d>0$.
\end{thm}

\section{Simulation results} \label{Sim} 
To illustrate the performance of our density estimation procedure, we performed simulations in different settings. 
For a grid of $T=101$ points $0=t_0<t_1<\ldots<t_{T}=1$ equispaced by $\Delta t=1/(T-1)$, we generated data 
$Y_i(t_k)=\sum_{j=1}^{J} \sqrt \lambda_j Z_{ik}\ \phi_j(t_k) +\sigma\, W_i(t_k)$, 
where the $Z_{ik}$'s are i.i.d., each $Z_{ik}$ is the average of the two independent $U[-.1,.1]$ random variables,   $W_i(t_0)=0$ and, for $k=1,\ldots,T$, $W_i(t_k)=W_i(t_{k-1}) + \epsilon_{ik}$, 
where the $\epsilon_{ik}$'s are i.i.d.~$\sim N(0,\Delta t)$.
We considered five settings: 
(i) $J=20$, $\sigma=0.1$, $\lambda_j=\exp(-j)$ and $\phi_j(t)=\sqrt 2\sin (\pi t j)$; (ii) same as (i) but with $J=40$; (iii) same as (ii) but with $\sigma=0.075$; (iv)  same as (i) but with $\sigma=0.075$, $\phi_j(t)=\sqrt 2\cos (\pi t j)\kappa(t)$, $\kappa(t)=2\exp(10 t)/\{1+\exp(10 t)\}-1$;  (v) same as (i) but with  $\sigma=0.075$,  $\phi_j(t)=\sqrt 2\sin (\pi t j)\kappa(t)$.

 In each case we generated $B=200$ samples of $Y_i(t_k)$'s, of sizes $n=500$, $1000$, $2000$ and $5000$. Then, for $b=1,\ldots,B$, using the $b$th sample of $Y_i(t_k)$'s, we computed our density estimator $\hat{f}_Y^{[m,K]}(V)$ at \eqref{eq:est} for $10^4$ functions $V$ generated from the same distribution as $\sigma W$, where $m$ and $K$ were chosen by cross-validation by minimisation of \eqref{eq:CV} and where we took the weight function $\omega_K(x)=1-x/(K+1)$. The basis functions $\varphi_j$ were computed as in Section~\ref{basis} with $M=20$ and $\psi_j(t)=\sqrt 2 \sin(\pi tj)$; we denote by DM the resulting estimator. Each time the $m$ and $K$ selected by CV produced a negative estimator $\hat f_Y(v)$  for a new data curve $v$, for that curve $v$ we repeatedly replaced, $K$ by $K-1$ and $m$ and $m-1$ until the resulting value of $(m,K)$ was such that $\hat f_Y(v)>0$. 

In each case we also computed the estimator of Dabo-Niang~(2004a) with our adaptive basis of  $\varphi_j$'s, which we denote by DN. We chose $K$ by minimisation of the cross-validation criterion at \eqref{eq:CV}, replacing there our estimator by this estimator and $(m,K)$ by $K$. As for our estimator, each time the selected value of $K$ produced a negative estimator for a new curve $v$, we replaced, for that curve $v$, $K$ by the largest value smaller or equal to $K$ which produced a positive estimator.

\begin{table}[h]
\small
	\centering
\caption{Simulation results for density estimation: $10^4 \times$ median [first quartile, second quartile] of $2\times 10^6$ values of the SE.}
\begin{tabular}{cc|llll}
Model&Method&\multicolumn{1}{ c }{$n=500$}&\multicolumn{1}{ c }{$n=1000$}&\multicolumn{1}{ c }{$n=2000$}&\multicolumn{1}{ c }{$n=5000$}\\
 \hline

(i)&DM&635[145,2242]&492[120,1660]&395[103,1252]&316[86,953]\\
&DN&891[171,4122]&800[166,3439]&664[125,2970]&527[100,2271]\\
(ii)&DM&683[152,2427]&506[123,1732]&409[108,1293]&343[94,1051]\\
&DN&911[179,4133]&823[168,3568]&659[124,2990]&544[101,2420]\\
(iii)&DM&1134[237,4538]&898[188,3529]&813[175,3237]&784[165,3197]\\
&DN&1375[209,8046]&1200[186,7325]&1081[174,6611]&1025[177,5574]\\
(iv)&DM&908[194,3788]&801[172,3158]&744[154,3135]&590[124,2399]\\
&DN&1468[232,8351]&1151[183,6878]&1097[190,6514]&1052[196,5460]\\
(v)&DM&849[187,3287]&751[163,2812]&654[143,2500]&565[122,2273]\\
&DN&1097[170,6389]&1024[172,5817]&914[160,5133]&865[160,4309]\\

 \hline   
  \end{tabular}
  \label{table:SIM}
\end{table}

\begin{table}[h]
\small
	\centering
\caption{Average computational time (in seconds) for computing one density estimator (including the CV choice of smoothing parameters).}
\begin{tabular}{cc|cccc}
Model&Method&\multicolumn{1}{ c }{$n=500$}&\multicolumn{1}{ c }{$n=1000$}&\multicolumn{1}{ c }{$n=2000$}&\multicolumn{1}{ c }{$n=5000$}\\
 \hline

(i)&DM&94&114&130&198\\
&DN&42&46&54&77\\
(ii)&DM&95&113&135&200\\
&DN&49&55&68&96\\
(iii)&DM&102&116&138&218\\
&DN&50&53&71&97\\
(iv)&DM&104&110&127&191\\
&DN&46&47&59&82\\
(v)&DM&91&130&125&182\\
&DN&41&47&65&100\\

 \hline   
  \end{tabular}
  \label{table:TIME}
\end{table}

We also considered the kernel density estimator  of Dabo-Niang~(2004b), which requires the choice of a bandwidth. To choose it in practice we considered several versions of cross-validation and a nearest-neighbour bandwidth version of the estimator. However we encountered major numerical issues with denominators getting too close to zero and did not manage to obtain reasonable results. Therefore we do not consider this estimator in our numerical work. 


The results of our simulations are summarised in Table~\ref{table:SIM} where, for each case and each sample size $n$ we present $10^4$ times the median and the first and third quartiles of the squared error $\textrm{SE}=\{\hat f_Y(V)-f_Y(V)\}^2$  computed for the $200\times 10^4$  $V$ values. As expected by the theory, both estimators improved as sample size increased, and overall our estimator worked significantly better than Dabo-Niang's~(2004a) estimator. In Table~\ref{table:TIME}, for our estimator and that of  Dabo-Niang~(2004a), we also show the average  time (in seconds and  averaged over 10 simulated examples)  required to compute one density estimator and its associated data-driven smoothing parameters. Recall that our estimator requires the choice by CV of two smoothing parameters $m$ and $K$ whereas  that of  Dabo-Niang~(2004a) requires to choose one smoothing parameter $K$. It is unsurprising then that our estimator requires longer computational time: this is the price to pay for the additional accuracy brought by choosing, in a data-driven way, two parameters instead of one.

\section{Proofs} \label{Proofs}
\subsection{Side results}\label{sec:side}

To prove \eqref{eq:PC.new.1}, note that since $\varphi_\ell\in L_2([0,1])$, we can write $\varphi_\ell=\sum_{j=1}^\infty \varphi_{\ell,j}  \psi_j$. Since $\Gamma_X \varphi_\ell=\lambda_\ell \varphi_\ell$, we deduce that
$ \sum_{j,k=1}^\infty \varphi_{\ell,j} \, \langle \psi_k, \Gamma_X \psi_j\rangle\, \psi_k \, = \, \lambda_\ell \sum_{k=1}^\infty \varphi_{\ell,k}\, \psi_k$. Multiplying both sides of this equality by $\psi_k$ and taking the integral we obtain \eqref{eq:PC.new.1}.

To prove \eqref{eq:rem.eq1}, note that,  using Fubini's theorem and integration by parts, we have
\begin{align*}
\langle \psi_k, \Gamma_X \psi_j\rangle=& \int_0^1 \psi_k(t) \big(\Gamma_X \psi_j\big)(t) \,dt  = E \Big\{\int_0^1 \psi_k(t) X'(t) \,dt \, \int_0^1 \psi_j(s) X'(s)\, ds\Big\} \notag\\
=&  \int_0^1 \psi_k'(t) \int_0^1 \big[E \big\{X(t) X(s)\big\}\big] \psi_j'(s) \,ds\, dt \notag \\
=&  \int_0^1 \psi_k'(t) \int_0^1 \Big(\big[E \big\{Y(t) Y(s)\big\}\big]  -\sigma^2\min(s,t) \Big)  \psi_j'(s) \,ds\, dt\notag \\
=&  \int_0^1 \psi_k'(t) \int_0^1 E \big\{Y(t) Y(s)\big\} \psi_j'(s)\, ds\, dt  -  \sigma^2 \int_0^1 \psi_k(t) \psi_j(t)\, dt\notag\\
=&   {\cal M}_{j,k}  -  \sigma^2\cdot 1\{j=k\}\,, 
\end{align*}
where we used the fact that $ \int_0^1 \psi_k'(t) \int_0^1 \min(s,t) \psi_j'(s)\, ds\, dt \, = \, \int_0^1 \psi_k(t) \psi_j(t) \,dt$.

In order to provide a more general/ abstract view of a major step (\ref{eq:T2.10}) in the proof of Theorem~\ref{T:2}, we mention that the supremum of a statistical risk $E_\theta \|\hat{\theta}-\theta\|^2$ over all $\theta\in\Theta$ is estimated from below by a Bayesian risk with respect to some a-priori distribution $Q$ on the parameter space $\Theta$. Therein $\Theta$ is a subset of a separable Hilbert space with the norm $\|\cdot\|$. Moreover impose that the data distribution has the density $f(\theta;\cdot)$ with respect to some dominating $\sigma$-finite measure $\mu$ on the action space $\Omega$. In order to calculate the smallest Bayesian risk, consider the classical argument that
\begin{align} \nonumber
E_Q \, & E_\theta \, \|\hat{\theta}-\theta\|^2 \, = \, \iint \|\hat{\theta}(\omega) - \theta\|^2 f(\theta;\omega) d\mu(\omega) dQ(\theta) \\  \nonumber
& \, = \, \iint \big\|\big(\hat{\theta}(\omega) - \tilde{\theta}(\omega)\big) + \big(\tilde{\theta}(\omega) - \theta\big)\big\|^2 f(\theta;\omega) d\mu(\omega) dQ(\theta) \\  \nonumber
& \, \geq \,  \iint \|\tilde{\theta}(\omega) - \theta\|^2 f(\theta;\omega) d\mu(\omega) dQ(\theta) \\ \label{eq:side10} & \hspace{0.2cm} \, + \, 2 \int \Big\langle \hat{\theta}(\omega) - \tilde{\theta}(\omega) , \tilde{\theta}(\omega) \int f(\theta;\omega) dQ(\theta)  - \int \theta f(\theta;\omega) dQ(\theta) \Big\rangle d\mu(\omega)\,,
  \end{align}
where $\langle\cdot,\cdot\rangle$ denotes the inner product associated with $\|\cdot\|$ and the integrals inside the inner product may be understood as Bochner integrals. Putting 
$$ \tilde{\theta}(\omega) :=  \int \theta f(\theta;\omega) dQ(\theta) \, / \,  \int f(\theta;\omega) dQ(\theta)\,, $$
the last term in (\ref{eq:side10}) vanishes so that $\tilde{\theta}$ is the Bayes estimator of $\theta$ with respect to $Q$ and $\|\cdot\|^2$. Thus the minimal Bayesian risk (Bayesian risk of $\tilde{\theta}$) equals
\begin{align*}
& E_Q \, E_\theta \, \|\tilde{\theta}-\theta\|^2 \\ & \, = \, \iint \Big\|\int \theta' f(\theta';\omega) dQ(\theta') \, / \,  \int f(\theta'';\omega) dQ(\theta'') - \theta\Big\|^2 f(\theta;\omega) d\mu(\omega) dQ(\theta) \\
& \, = \, \int \Big\|\int \theta' f(\theta';\omega) dQ(\theta')\Big\|^2 \, \Big\{\int f(\theta'';\omega) dQ(\theta'')\Big\}^{-2} \, \int f(\theta;\omega) dQ(\theta) d\mu(\omega) \\
& \, - \, 2 \int \Big\langle \int \theta' f(\theta';\omega) dQ(\theta') , \int \theta f(\theta;\omega) dQ(\theta)\Big\rangle \Big\{\int f(\theta'';\omega) dQ(\theta'')\Big\}^{-1} d\mu(\omega) \\
& \, + \, \int \|\theta\|^2 \underbrace{\int f(\theta;\omega)  d\mu(\omega)}_{=\, 1} dQ(\theta)\,,
\end{align*}
so that
\begin{align*}
E_Q  E_\theta  \|\tilde{\theta}-\theta\|^2 & \, = \, \int \|\theta\|^2 dQ(\theta) \\ & \quad - \, \int \Big\|\int \theta' f(\theta';\omega) dQ(\theta')\Big\|^2  \big/ \Big\{\int f(\theta;\omega) dQ(\theta)\Big\} d\mu(\omega). 
\end{align*}
This corresponds to the lower bound on the minimax risk which is applied in (\ref{eq:T2.10}).

\subsection{Proof of Theorem~\ref{P:0}}\label{sec:proof_main}

Since the measure $P_V$ of $V_1$ is known, we can identify the measure $P_Y$ from the Radon-Nikodym derivative $f_Y = dP_Y/dP_V$. Suppose there exist two measures $P_X$ and $\tilde{P}_X$, each of which is a candidate for the true measure of $X_1$, and both of which lead to the same measure $P_Y$ of $Y_1 = X_1 + V_1$. Consider the functional characteristic functions $\psi_X$, $\tilde{\psi}_X$ and $\psi_Y$, defined by
\begin{align*}
\psi_X(t) & \, = \, \int \exp\Big\{i \int_0^1 t(u) x(u) \,du\Big\}\, dP_X(x)\,, \\
\tilde{\psi}_X(t) & \, = \, \int \exp\Big\{i \int_0^1 t(u) x(u)\, du\Big\} \,d\tilde{P}_X(x)\,, \\
\psi_Y(t) & \, = \,  \int \exp\Big\{i \int_0^1 t(u) x(u)\, du\Big\}\, dP_Y(x)\,, \\ 
\psi_V(t) & \, = \,  \int \exp\Big\{i \int_0^1 t(u) x(u)\, du\Big\} \,dP_V(x) \\
& \, = \, \exp\Big\{-\frac12 \sigma^2 \int_0^1 \int_0^1 t(u) \min(u,u') t(u') \, du \, du'\Big\}\,, 
\end{align*}
for any $t \in L_2([0,1])$. It follows from the independence of $X_1$ and $V_1$ that
$ \psi_X(t) \cdot \psi_V(t)  =  \psi_Y(t)  =  \tilde{\psi}_X(t)\cdot \psi_V(t)$  for all $t\in L_2([0,1])$. 
 Since $\psi_V$ does not vanish anywhere, the above equality implies that $\psi_X=\tilde{\psi}_X$. Now, for $u\in [0,1]$, we put
$ t(u)  =  t_h(u) =h^{-1} \sum_{j=1}^{2^m-1} \tau_j\cdot K\big\{(u-j/2^m)/h\big\} $, 
where $m>0$ is integer, the $\tau_j$'s are real coefficients, and for a bandwidth parameter $h \in (0,2^{-m}]$ and a kernel function $K:\mathbb{R}\to\mathbb{R}$, which is non-negative, continuous, supported on the interval $[-1,1]$ and integrates to one. 

For any fixed $m$ and $\tau_j$, $j=1,\ldots,2^m-1$, we have
$ \lim_{h\to 0} \int_0^1 t_h(u) x(u) \,du \, = \, \sum_{j=1}^{2^m-1} \tau_j\cdot x(j/2^m)$, 
for any $x \in C_0([0,1])$. By dominated convergence it follows that $\psi_X(t_h) = \tilde{\psi}_X(t_h)$ tend to the characteristic functions of the random vector $$X_1^{[m]} = \big(X_1(1/2^m),\ldots,X_1((2^m-1)/2^m)\big)$$ at $\tau = (\tau_1,\ldots,\tau_{2^m-1})$ under the probability measure $P_X$ and $\tilde{P}_X$, respectively, as $h\downarrow 0$. Since $\tau$ can be chosen to be any vector in $\mathbb{R}^{2^m-1}$, the above mentioned characteristic functions are equal. It is well known that the characteristic function of any random vector in $\mathbb{R}^{2^m-1}$ determines its distribution uniquely so that the distributions of $X_1^{[m]}$ under the basic measure $P_X$, on the one hand, and $\tilde{P}_X$, on the other hand, are identical. Thus, for some arbitrary $s \in C_0([0,1])$, we have
\begin{align} \nonumber P_X\big(\big\{&x\in C_{0,0}([0,1])\, : \, x(j/2^m) \leq s(j/2^m),\, \forall j=1,\ldots,2^m-1\big\}\big) \\  \label{eq:Proof.0.0} & \, = \, \tilde{P}_X\big(\big\{x\in C_{0,0}([0,1])\, : \, x(j/2^m) \leq s(j/2^m),\, \forall j=1,\ldots,2^m-1\big\}\big)\,. \end{align}

The countable set ${\cal Q} = \bigcup_{m\in \mathbb{N}} \{k/2^m : k=1,\ldots,2^m-1\}$ is dense in the interval $[0,1]$. Hence the following events coincide:
\begin{align*}
\big\{x&\in C_{0,0}([0,1]) :  \, x(u) \leq s(u),\, \forall u\in [0,1]\big\} \\ &  =  \big\{x\in C_{0,0}([0,1]) : \, x(u) \leq s(u),\, \forall u\in {\cal Q}\big\} \\ & \, = \, \bigcap_{m\in \mathbb{N}} \big\{x\in C_{0,0}([0,1])\, : \, x(j/2^m) \leq s(j/2^m),\, \forall j=1,\ldots,2^m-1\big\}.
\end{align*}
Therefore we obtain that
\begin{align} \nonumber 
 P_X&\big(\big\{x\in C_{0,0}([0,1])\, : \, x(u) \leq s(u),\, \forall u\in [0,1]\big\}\big) \\ & \label{eq:Proof.0.1} \, = \, \lim_{m\to\infty}  P_X\big(\big\{x\in C_{0,0}([0,1])\, : \, x(j/2^m) \leq s(j/2^m),\, \forall j=1,\ldots,2^m-1\big\}\big)\,. \end{align}
The corresponding equality holds true for the measure $\tilde{P}_X$.

Combining (\ref{eq:Proof.0.0}) and (\ref{eq:Proof.0.1}) we deduce that
\begin{align*}
 P_X\big(\big\{x\in C_{0,0}([0,1])\,& : \, x(u) \leq\, s(u),\, \forall u\in [0,1]\big\}\big) \\ & \, = \, \tilde{P}_X\big(\big\{x\in C_{0,0}([0,1])\, : \, x(u) \leq s(u),\, \forall u\in [0,1]\big\}\big)\,, \end{align*}
for any $s \in C_0([0,1])$. The system of the sets 
$$ \big\{x\in C_{0,0}([0,1])\, : \, x(u) \leq s(u),\, \forall u\in [0,1]\big\},\qquad s \in C_0([0,1])\,, $$
is stable with respect to intersection and generates the Borel $\sigma$-field $\mathfrak{B}(C_{0,0}([0,1]))$. Therefore, by the uniqueness theorem for measures, we conclude that $P_X = \tilde{P}_X$.

\subsection{Proof of Proposition \ref{P:privacy}}
Let $x$ and $\tilde{x}$ be two realizations of the functional random variable $X$. Thanks to Assumptions 1 and 2, we may impose that $x(0)=\tilde{x}(0)=0$ and that $\max\{\|x'\|_2, \|\tilde{x}'\|_2\} \leq C_{X,1}$. For any $t_1,\ldots,t_n \in [0,1]$ we introduce the vector $F = \big(x(t_j) - \tilde{x}(t_j)\big)^T_{j=1,\ldots,n}$ and the matrix $M = \big\{E W(t_j) W(t_k)\big\}_{j,k =1,\ldots,n}$. According to Proposition~7 in Hall et al.~(2013), in order to prove privacy it suffices to show that
$\big|M^{-1/2} F\big| \, \leq \, \sigma \alpha / c(\beta)$, 
where we may put $c(\beta) = \sqrt{2 \log(2/\beta)}$ according to Proposition~3 in Hall et al.~(2013). Without any loss of generality we assume that $t_1 \leq \cdots \leq t_n$ since
$ \big|(P M P^T)^{-1/2} (PF)\big|^2 =  F^T M^{-1} F =  \big|M^{-1/2} F\big|^2$, 
for any $n\times n$-permutation matrix $P$. Then,
$$ M = \begin{pmatrix} t_1 & t_1 & t_1 & \cdots & t_1 \\
					   t_1 & t_2 & t_2 & \cdots & t_2 \\
					   t_1 & t_2 & t_3 & \cdots & t_3 \\
					   \vdots & \vdots & \vdots & \vdots & \vdots \\
					   t_1 & t_2 & t_3 & \cdots & t_n \end{pmatrix}\,. $$

Writing $\Delta_j = (F_j - F_{j-1})/(t_j - t_{j-1})$ if $t_j>t_{j-1}$ and $x'(t_j) - \tilde{x}'(t_j)$ if $t_j = t_{j-1}$; and $Y_j = \Delta_j - \Delta_{j+1}$, 
where we set $F_0 = t_0 = 0$ and $\Delta_{n+1} = 0$, we consider that
\begin{align*}
\sum_{l=1}^{k-1} t_l Y_l + \sum_{l=k}^n t_k Y_l & \, = \, \sum_{l=1}^{k-1} t_l (\Delta_l - \Delta_{l+1}) + t_k \Delta_k  \, = \, t_1 \Delta_1 + \sum_{l=2}^k (t_l + t_{l-1}) \Delta_l \, = \, F_k\,, \end{align*}
for all integer $k=1,\ldots,n$ so that $M Y = F$, where $Y = (Y_j)^T_{j=1,\ldots,n}$. We deduce that the left hand side of the above system of equations equals
\begin{align} \label{eq:priv2}
F^T M^{-1} F &  =  F^T Y  =  \sum_{j=1}^n F_j \Delta_j - \sum_{j=1}^n F_j \Delta_{j+1}  = \sum_{j=1}^n \frac{(F_j - F_{j-1})^2}{t_j - t_{j-1}}\,. \end{align}
As 
$ F_j - F_{j-1} \, = \, \int_{t_{j-1}}^{t_j} \big\{x'(t) - \tilde{x}'(t)\big\} dt$ for $j=1,\ldots,n$, 
the Cauchy-Schwarz inequality in $L_2([0,1])$ yields that (\ref{eq:priv2}) has the upper bound
$$ \sum_{j=1}^n \int_{t_{j-1}}^{t_j} \big|x'(t) - \tilde{x}'(t)\big|^2 dt \, \leq \, 4 C_{X,1}^2\,, $$
which completes the proof of the proposition. \hfill $\square$ \\

\noindent {\it Proof of Lemma~\ref{L:1}}: (a)\; Expanding $X_1'$ in the orthonormal basis $\{\varphi_j\}_j$ we get
$$ \int_0^1 X_1'(t) \,dV_1(t)  \, = \, \sum_{j=1}^\infty \langle X_1', \varphi_j \rangle \cdot \beta'_{V_1,j}\,, \ \ 
\|X_1'\|_2^2  \, = \, \sum_{j=1}^\infty \big|\langle X_1' , \varphi_j\rangle\big|^2\,, $$
where the infinite sums should be understood as mean squared limits. 

Since, for any integer $m$, $\mathfrak{A}_{m}$ is a subset of the $\sigma$-field generated by $V_1$, we have that
\begin{align} \nonumber
E&\big\{f_Y(V_1) \mid \mathfrak{A}_{m}\big\} \, = \, E\Big\{ \exp\Big(\frac1{\sigma^2} \int_0^1 X_1'(t)\, dV_1(t) -\frac1{2\sigma^2} \int_0^1 |X_1'(t)|^2\, dt\Big) \Big| \mathfrak{A}_m\Big\} \\ \nonumber
& \, = \, E\Big\{ \exp\Big(\frac1{\sigma^2} \sum_{j=1}^\infty \langle X_1',\varphi_j\rangle \cdot \beta'_{V_1,j}\Big) \cdot \exp\big(- \|X_1'\|_2^2 / (2 \sigma^2)\big) \Big| \mathfrak{A}_m\Big\} \\ \nonumber
& \, = \,  E\Big\{ \exp\Big(\frac1{\sigma^2} \sum_{j=1}^m \langle X_1',\varphi_j\rangle \cdot \beta'_{V_1,j}\Big)\cdot \exp\big(- \|X_1'\|_2^2 / (2\sigma^2)\big) \\ \nonumber
& \hspace{3.5cm} \cdot  E\Big\{ \exp\Big(\frac1{\sigma^2} \sum_{j>m} \langle X_1',\varphi_j\rangle \cdot \beta'_{V_1,j}\Big) \Big| \mathfrak{A}_m, X_1'\Big\} \Big| \mathfrak{A}_m\Big\}  \\  \nonumber
& \, = \,  E\Big\{ \exp\Big(\frac1{\sigma^2} \sum_{j=1}^m \langle X_1',\varphi_j\rangle \cdot \beta'_{V_1,j}\Big)\cdot \exp\big(- \|X_1'\|_2^2 / (2\sigma^2)\big) \\ \label{eq:L1.1} & \hspace{4cm} 
 \cdot \prod_{j>m} E\Big\{ \exp\Big(\frac1{\sigma^2} \langle X_1',\varphi_j\rangle \cdot \beta'_{V_1,j}\Big) \Big| X_1'\Big\} \Big| \mathfrak{A}_m\Big\}
\end{align}
holds true almost surely. 

Applying, to the last term in (\ref{eq:L1.1}), the fact that
\begin{equation} \label{eq:expGauss} 
E \big\{\exp\big(t \delta\big)\big\} \, = \, \exp\big(t^2/2\big)\,, \end{equation}
for all $\delta \sim N(0,1)$ and $t\in \mathbb{R}$, we deduce that
\begin{align*} 
E&\big\{f_Y(V_1) \mid \mathfrak{A}_{m} \big\} \, = \,  E\Big\{ \exp\Big(\frac1{\sigma^2} \sum_{j=1}^m \langle X_1',\varphi_j\rangle \cdot \beta'_{V_1,j}\Big)\cdot \exp\big(- \|X_1'\|_2^2 / (2\sigma^2)\big) \\ 
& \hspace{6.2cm} \cdot \exp\Big( \frac1{2\sigma^2}  \sum_{j>m} \big|\langle X_1',\varphi_j\rangle\big|^2\Big) \Big| \mathfrak{A}_m\Big\} \\
& \, = \,  E\Big\{ \exp\Big(\frac1{\sigma^2} \sum_{j=1}^m \langle X_1',\varphi_j\rangle \cdot \beta'_{V_1,j}\Big)\cdot \exp\Big(- \frac1{2\sigma^2}  \sum_{j=1}^m \big|\langle X_1',\varphi_j\rangle\big|^2\Big) \Big| \mathfrak{A}_m\Big\} \\
 & \, = \, \int \exp\Big(\frac1{\sigma^2} \sum_{j=1}^m \langle x',\varphi_j\rangle \cdot \beta'_{V_1,j} - \frac1{2\sigma^2}  \sum_{j=1}^m \big|\langle x',\varphi_j\rangle\big|^2\Big)\, dP_X(x) \\
& \, = \, f_Y^{[m]}(\beta'_{V_1,1},\ldots,\beta'_{V_1,m})
\end{align*}
almost surely, which completes the proof of part (a). \\
(b)\; Using the result of part (a) we have
\begin{align} \nonumber 
E & \big|f_Y^{[m]}(\beta'_{V_1,1},\ldots,\beta'_{V_1,m}) - f_Y(V_1)\big|^2 \\ \nonumber & \, = \, E \, \big[E\big\{\big|f_Y^{[m]}(\beta'_{V_1,1},\ldots,\beta'_{V_1,m}) - f_Y(V_1)\big|^2 \mid \mathfrak{A}_m\big\}\big] \\ \label{eq:L1.2}
& \, = \, E \,\big[ \mbox{var}\big\{f_Y(V_1) \mid \mathfrak{A}_m\big\}\big]\,.
\end{align}
Using Fubini's theorem, we get
\begin{align} \nonumber
E \, &\big[ \mbox{var}\big\{f_Y(V_1) \mid \mathfrak{A}_m\big\} \big]\\  
\nonumber & \, = \, E\, \Big\{ \mbox{var}\Big(\int \exp\Big(\frac1{\sigma^2} \sum_{j=1}^\infty \langle x',\varphi_j\rangle \cdot \beta'_{V_1,j}\Big)\cdot \exp\big(- \|x'\|_2^2 / (2\sigma^2)\big) \,dP_{X}(x) \Big| \mathfrak{A}_m\Big)\Big\} \\  \nonumber
& \, = \, \iint \exp\big(- \{\|x_1'\|_2^2 + \|x_2'\|_2^2\} / (2\sigma^2)\big) \, E\Big\{ \exp\Big(\frac1{\sigma^2} \sum_{j=1}^m \langle x_1' + x_2',\varphi_j\rangle \cdot \beta'_{V_1,j}\Big)\Big\} \\ \nonumber & \hspace{0.2cm} \cdot \mbox{cov}\Big\{ \exp\Big(\frac1{\sigma^2} \sum_{j>m} \langle x_1',\varphi_j\rangle \cdot \beta'_{V_1,j}\Big) ,  \exp\Big(\frac1{\sigma^2} \sum_{j>m} \langle x_2',\varphi_j\rangle \cdot \beta'_{V_1,j}\Big)\Big\}\, \\  \label{eq:L1.3} & \hspace{8cm} dP_{X}(x_1) \,dP_{X}(x_2)\,.
\end{align}
Using (\ref{eq:expGauss}) again we deduce that
$$ E \big\{\exp\big(\sigma^{-2} \sum_{j=1}^m \langle x_1' \allowbreak+ x_2',\varphi_j\rangle \cdot \beta'_{V_1,j}\big)\big\}  =  \exp\big\{ \sum_{j=1}^m \big|\langle x_1' + x_2',\varphi_j\rangle\big|^2/(2 \sigma^2)\big\}\,, $$
and that
\begin{align*}
&\mbox{cov}\Big\{ \exp\Big(\frac1{\sigma^2} \sum_{j>m} \langle x_1',\varphi_j\rangle \cdot \beta'_{V_1,j}\Big) ,  \exp\Big(\frac1{\sigma^2} \sum_{j>m} \langle x_2',\varphi_j\rangle \cdot \beta'_{V_1,j}\Big)\Big\} \\
& \, = \, E \Big\{ \exp\Big(\frac1{\sigma^2} \sum_{j>m} \langle x_1'+x_2',\varphi_j\rangle \cdot \beta'_{V_1,j}\Big) \Big\}\\ & 
\, - \,  \Big[E\Big\{ \exp\Big(\frac1{\sigma^2} \sum_{j>m} \langle x_1',\varphi_j\rangle \cdot \beta'_{V_1,j}\Big)\Big\}\Big]
\cdot \Big[E\Big\{ \exp\Big(\frac1{\sigma^2} \sum_{j>m} \langle x_2',\varphi_j\rangle \cdot \beta'_{V_1,j}\Big)\Big\}\Big] \\
& \, =  \exp\Big(\frac1{2\sigma^2} \sum_{j>m} \big|\langle x_1'+x_2',\varphi_j\rangle\big|^2\Big)  -  \exp\Big(\frac1{2\sigma^2} \sum_{j>m} \big\{\big|\langle x_1',\varphi_j\rangle\big|^2 + \big|\langle x_2',\varphi_j\rangle\big|^2\big\}\Big)\,.
\end{align*}
Plugging these equalities into (\ref{eq:L1.3}) we conclude that
\begin{align} \nonumber
& E \,\big[ \mbox{var}\big\{f_Y(V_1) \mid \mathfrak{A}_m\big\} \big]  
 = \iint \Big[\exp\big\{- (\|x_1'\|_2^2 + \|x_2'\|_2^2 - \|x_1' + x_2'\|_2^2) / (2\sigma^2)\big\} \\  \nonumber 
&  - \exp\Big\{- \frac1{2\sigma^2} \sum_{j=1}^m \big(\big|\langle x_1',\varphi_j\rangle\big|^2 + \big|\langle x_2',\varphi_j\rangle\big|^2 - \big|\langle x_1'+x_2',\varphi_j\rangle\big|^2\big)\Big\}\Big]  dP_{X}(x_1) dP_{X}(x_2) \\ \label{eq:L1.4} 
& = \iint \Big\{\exp\big(\langle x_1',x_2'\rangle / \sigma^2\big) - \exp\Big( \frac1{\sigma^2} \sum_{j=1}^m \langle x_1',\varphi_j\rangle  \langle x_2',\varphi_j\rangle\Big)\Big\} \,dP_{X}(x_1) \,dP_{X}(x_2)\,.
\end{align}

Let $X_2$ denote an independent copy of $X_1$. Then  (\ref{eq:L1.4}) satisfies
\begin{align*}
E \Big\{\exp\big(\langle & X_1',X_2'\rangle / \sigma^2\big) - \exp\Big( \frac1{\sigma^2} \sum_{j=1}^m \langle X_1',\varphi_j\rangle  \langle X_2',\varphi_j\rangle\Big)\Big\} \\
& \, \leq \, \frac1{\sigma^2} E \Big| \sum_{j>m} \langle X_1' , \varphi_j \rangle \langle X_2',\varphi_j\rangle \Big|\cdot \exp\big(C_{X,1}^2 / \sigma^2\big) \\
& \, \leq \, \frac1{\sigma^2} \cdot \exp\big(C_{X,1}^2 / \sigma^2\big) \cdot \Big(\sum_{j,j'>m} \big|\langle \varphi_j , \Gamma_X \varphi_{j'} \rangle\big|^2\Big)^{1/2}\,, 
\end{align*}
where we used the mean value theorem and the Cauchy-Schwarz inequality.

\subsection{Proof of Theorem~\ref{P:1}}
Let $V \sim P_V$ denote a functional random variable which is independent of $X_1,\ldots,X_n$ and $W_1,\ldots,W_n$,  and let
$ \beta'_{V,j} \, = \, \int_0^1 \varphi_j(t)\, dV(t)$. 
Since $\hat{f}_Y^{[m,K]}(V)$ is measurable in the $\sigma$-field generated by ${\beta'_{V,1}},\ldots,{\beta'_{V,m}},Y_1,\ldots,Y_n$ and as
\begin{align*} f_Y^{[m]}({\beta'_{V,1}},\ldots,{\beta'_{V,m}}) & =  E\big\{f_Y(V) \mid {\beta'_{V,1}},\ldots,{\beta'_{V,m}}\big\} \\ & =  E\big\{f_Y(V) \mid {\beta'_{V,1}},\ldots,{\beta'_{V,m}},Y_1,\ldots,Y_n\big\}\,, \end{align*}
a.s., by Lemma~\ref{L:1}(a), we have
\begin{align}
{\cal R}&\big(\hat{f}_Y^{[m,K]},f_Y\big)  =  E \big|\hat{f}_Y^{[m,K]}(V) - f_Y(V)\big|^2\notag\\
& =  E \big[ E\big\{\big|\hat{f}_Y^{[m,K]}(V) - f_Y(V)\big|^2 \mid {\beta'_{V,1}},\ldots,{\beta'_{V,m}},Y_1,\ldots,Y_n\big\}\big] \nonumber \\  \nonumber
&  =   E \, \Big[\mbox{var}\big\{f_Y(V) \mid {\beta'_{V,1}},\ldots,{\beta'_{V,m}},Y_1,\ldots,Y_n\big\}\Big] \\ \nonumber & \hspace{5cm} + \, E \,\big|\hat{f}_Y^{[m,K]}(V) - f_Y^{[m]}({\beta'_{V,1}},\ldots,{\beta'_{V,m}})\big|^2 \\  \nonumber
&  =   E \, \Big[\mbox{var}\big\{f_Y(V) \mid {\beta'_{V,1}},\ldots,{\beta'_{V,m}}\big\}\Big] \, + \, E \, \big|\hat{f}_Y^{[m,K]}(V) - f_Y^{[m]}({\beta'_{V,1}},\ldots,{\beta'_{V,m}})\big|^2 \\  \label{eq:P1.1} & \leq {\cal D} \, + \, E \big|\hat{f}_Y^{[m,K]}(V) - f_Y^{[m]}({\beta'_{V,1}},\ldots,{\beta'_{V,m}})\big|^2\,,
\end{align}
using also Lemma~\ref{L:1}(b). Using the definition (\ref{eq:est}) of the estimator $\hat{f}_Y^{[m,K]}$ and Parseval's identity with respect to the orthonormal basis of the $H_{k_1,\ldots,k_m}$ in $L_{2,g_1}(\mathbb{R}^m)$, we get

\begin{align} \nonumber
& E \big|\hat{f}_Y^{[m,K]}(V) - f_Y^{[m]}({\beta'_{V,1}},\ldots,{\beta'_{V,m}})\big|^2  =  E \Big\|\sum_{k_1,\ldots,k_m\geq 0} 1\{k_1+\cdots+k_m\leq K\} \\
\notag&\hspace{1.7cm} \cdot\, \frac1n \sum_{j=1}^n H_{k_1,\ldots,k_m}(\beta'_{Y_j,1}/\sigma,\ldots,\beta'_{Y_j,m}/\sigma) \cdot H_{k_1,\ldots,k_m}  -  f_Y^{[m]}(\sigma\cdot)\Big\|_{g_1}^2 \\ \nonumber
& \, = \, \sum_{k_1,\ldots,k_m\geq 0} 1\{k_1+\cdots+k_m\leq K\} \\
  \label{eq:MISE1} & \hspace{0.7cm} \cdot E \Big|\frac1n \sum_{j=1}^n H_{k_1,\ldots,k_m}(\beta'_{Y_j,1}/\sigma,\ldots,\beta'_{Y_j,m}/\sigma) - \big\langle f_Y^{[m]}(\sigma\cdot) , H_{k_1,\ldots,k_m} \big\rangle_{g_1}\Big|^2 \, + \, {\cal B}\,.
\end{align}
Since, from \eqref{eq:proj},
$$ E \Big\{\frac1n \sum_{j=1}^n H_{k_1,\ldots,k_m}(\beta'_{Y_j,1}/\sigma,\ldots,\beta'_{Y_j,m}/\sigma)\Big\} \, = \, \big\langle f_Y^{[m]}(\sigma\cdot) , H_{k_1,\ldots,k_m} \big\rangle_{g_1}\,, $$
it follows that
\begin{align*}
E& \Big|\frac1n \sum_{j=1}^n H_{k_1,\ldots,k_m}(\beta'_{Y_j,1}/\sigma,\ldots,\beta'_{Y_j,m}/\sigma) - \big\langle f_Y^{[m]}(\sigma\cdot) , H_{k_1,\ldots,k_m} \big\rangle_{g_1}\Big|^2 \\
& \, = \, \mbox{var}\Big(\frac1n \sum_{j=1}^n H_{k_1,\ldots,k_m}(\beta'_{Y_j,1}/\sigma,\ldots,\beta'_{Y_j,m}/\sigma)\Big) \\ & \, \leq \, \frac1n\cdot E H_{k_1,\ldots,k_m}^2(\beta'_{Y_1,1}/\sigma,\ldots,\beta'_{Y_1,m}/\sigma)\,.
\end{align*}

Using the fact that the Hermite polynomials form an Appell sequence (see e.g.~Appell,~1880) we deduce that
\begin{align} \nonumber
E &\big\{\, H_{k_1,\ldots,k_m}^2(\beta'_{Y_1,1}/\sigma,\ldots,\beta'_{Y_1,m}/\sigma) \big\}\, = \, E \Big[\prod_{l=1}^m E \big\{ H_{k_l}^2(\beta'_{X_1,l}/\sigma+\beta'_{V,1,l}/\sigma) \mid X_1'\big\} \Big]\\ \nonumber
&  =  E\Big[ \prod_{l=1}^m \frac1{k_l!} E\Big\{\Big(\sum_{j=0}^{k_l} \sqrt{j!} {k_l \choose j} \sigma^{j-k_l} {\beta'_{X_1,l}}^{k_l-j} H_j(\beta'_{V,1,l}/\sigma)\Big)^2\mid X_1'\Big\} \Big]\\ \nonumber
&  =  E \Big[\prod_{l=1}^m \Big\{\frac1{k_l!} \sum_{j,j'=0}^{k_l} \sqrt{j!j'!} {k_l \choose j} {k_l \choose j'} \sigma^{j+j'-2k_l} {\beta'_{X_1,l}}^{2k_l-j-j'} \\ \nonumber & \hspace{7cm} \cdot \int H_j(t) H_{j'}(t) \frac1{\sqrt{2\pi}} e^{-\frac{t^2}{2}} \,dt\Big\}\Big] \\ \nonumber
&  =  E \Big[\prod_{l=1}^m \Big\{\frac1{k_l!} \sum_{j=0}^{k_l} j! {k_l \choose j}^2 \sigma^{2(j-k_l)} {\beta'_{X_1,l}}^{2(k_l-j)} \Big\} \Big] \\ \nonumber & \, = \, E \Big[\prod_{l=1}^m \Big\{\sum_{j=0}^{k_l} \frac1{j!} {k_l \choose j} \sigma^{-2j} {\beta'_{X_1,l}}^{2j} \Big\}\Big] \\ \label{eq:H2}
&  \leq  E \Big[\prod_{l=1}^m \big\{1 + (\beta'_{X_1,l}/\sigma)^2\big\}^{k_l}\Big] \, \leq \, \exp\big(K C_{X,1}^2/\sigma^2\big)\,,  
\end{align}
using the orthonormality of the $H_{k_1,\ldots,k_m}$ with respect to $\langle\cdot,\cdot\rangle_{g_1}$. Using elementary arguments from combinatorics, we also have
$$ \# \big\{(k_1,\ldots,k_m) \in \mathbb{N}_0 \, : \, k_1+\cdots+k_m \leq K\big\} \, = \, {K+m \choose K}\,. $$
Combined with \eqref{eq:H2}, this implies that the first term in (\ref{eq:MISE1}) is bounded from above by ${\cal V}$. Combining this with the other derivations above completes the proof of the theorem.

\subsection{Proof of Theorem~\ref{T:1}}

The next lemma gives an upper bound on the term ${\cal B}$ defined in Theorem~\ref{P:1}. It will be used to prove the theorem.
\begin{lem} \label{L:smooth}
Under Assumptions 1 and 2, the term ${\cal B}$ in Theorem~\ref{P:1} satisfies
$ {\cal B} \, =\, {\cal O}\big\{ (C_{X,1}/\sigma)^{2K} (2C_{X,1}/\sigma + \sqrt{2})^{2K} / (K+1)!\big\} $, 
where the constants contained in ${\cal O}(\cdots)$ only depend on $C_{X,1}$ and $\sigma$.
\end{lem}

\noindent {\it Proof of Lemma~\ref{L:smooth}}: By Taylor expansion we can write $f_Y^{[m]} = T_{m,K} + R_{m,K}$ where $R_{m,K}$ is a remainder term that will be treated below, and 
\begin{align*}
& T_{m,K}(s_1,\ldots,s_m)  \, = \, E \Big\{\sum_{k=0}^K \frac{1}{k!} \sigma^{-2k} \Big(\sum_{j=1}^m \beta'_{X_1,j} \cdot s_j\Big)^k \exp\Big(-\frac1{2\sigma^2} \sum_{j=1}^m {\beta'_{X_1,j}}^2\Big)\Big\} \\
& \quad = \, \sum_{k=0}^K \frac1{k!} \sigma^{-2k} \sum_{j_1,\ldots,j_k=1}^m \ \Big(\prod_{l=1}^k s_{j_l}\Big) \cdot E \Big\{\Big(\prod_{l=1}^k \beta'_{X_1,j_l}\Big) \exp\Big(-\frac1{2\sigma^2} \sum_{j=1}^m {\beta'_{X_1,j}}^2\Big)\Big\}\,.
\end{align*}
(Assumption 2 guarantees integrability of the above terms). 
Now $T_{m,K}(\sigma\cdot)$ is an $m$-variate polynomial of degree $\leq K$, so that $T_{m,K}(\sigma\cdot)$ is contained in the linear subspace ${\cal H}_{m,K}$ of $L_{2,g_1}(\mathbb{R}^m)$. It follows from there that
\begin{equation} \label{eq:upplow} {\cal B} \, \leq \, \big\| R_{m,K}(\sigma\cdot) \big\|_{g_1}^2\,. \end{equation}

Next, using the Lagrange representation, the remainder term $R_{m,K}$  has the following upper bound:
\begin{align*}
\big|&R_{m,K}(s_1,\ldots,s_m)\big|  \, \leq \, \frac1{(K+1)!} E\Big[\Big|\frac1{\sigma^2} \sum_{j=1}^m \beta'_{X_1,j} \cdot s_j\Big|^{K+1} \\ & \hspace{2.5cm} \cdot \max\Big\{\exp\Big(-\frac1{\sigma^2} \sum_{j=1}^m \beta'_{X_1,j} \cdot s_j\Big),1\Big\}\exp\Big(-\frac1{2\sigma^2} \sum_{j=1}^m {\beta'_{X_1,j}}^2\Big)\Big]\,,
\end{align*}
so that, by Jensen's inequality,
\begin{align} \nonumber
\big\| R_{m,K}(\sigma\cdot)& \big\|_{g_1}^2 \leq  \frac1{[(K+1)!]^2} E\Big[\Big|\frac1{\sigma^2} \sum_{j=1}^m \beta'_{X_1,j} \cdot \beta'_{V,j}\Big|^{2(K+1)} \\ \label{eq:RmK} & \hspace{0.4cm} \cdot \max\Big\{\exp\Big(-\frac2{\sigma^2} \sum_{j=1}^m \beta'_{X_1,j} \cdot \beta'_{V,j}\Big),1\Big\}\exp\Big(-\frac1{\sigma^2} \sum_{j=1}^m {\beta'_{X_1,j}}^2\Big)\Big]\,.
\end{align}

Conditionally on $X_1'$, the random variable $\sum_{j=1}^m \beta'_{X_1,j} \cdot \beta'_{V,j} / \sigma^2$ is normally distributed with mean $0$ and variance $\kappa_m^2 = \sum_{j=1}^m {\beta'_{X_1,j}}^2 / \sigma^2$. Thus, the right hand side of \eqref{eq:RmK} can be expressed as
$$ \frac1{\{(K+1)!\}^2} E \Big(\kappa_m^{2(K+1)} \exp\big(-\kappa_m^2\big) \, E  \big[\delta^{2(K+1)} \cdot \max\big\{\exp\big(2\kappa_m \delta\big),1\big\} \mid X_1'\big]\Big)\,, $$
where $\delta \sim N(0,1)$ and $X_1'$ are independent. Thus \eqref{eq:RmK} has the following upper bound:
\begin{align} \nonumber
 & \frac1{\{(K+1)!\}^2}  E \Big\{\kappa_m^{2(K+1)} \exp\big(-\kappa_m^2\big) \, E  \delta^{2(K+1)} \Big\}\\ \nonumber 
 & \hspace{1.5cm} + \, \frac1{\{(K+1)!\}^2} E\Big[ \kappa_m^{2(K+1)} \exp\big(-\kappa_m^2\big) \, E \big\{ \delta^{2(K+1)} \exp\big(2\kappa_m \delta\big)\mid X_1'\big\}\Big] \\  \nonumber
& \, = \, \frac1{\{(K+1)!\}^2} E \Big\{ \kappa_m^{2(K+1)} \exp\big(-\kappa_m^2\big)\Big\} \, 2^{K+1} \Gamma(K+3/2) / \sqrt{\pi} \\ \nonumber 
& \quad + \frac1{\{(K+1)!\}^2} E \Big\{\kappa_m^{2(K+1)} \exp\big(-\kappa_m^2\big)  \int s^{2(K+1)} \exp\big(2\kappa_m s - s^2/2\big) \,ds \Big\}/ \sqrt{2\pi}  \\ \nonumber
 & \, \leq \, {\cal O}\big\{ (2C_{X,1}/\sigma)^{2K} / (K+1)!\big\} \\ \nonumber 
& \quad + \frac1{\{(K+1)!\}^2} E\Big\{ \kappa_m^{2(K+1)} \exp\big(\kappa_m^2\big)  \int (s+2\kappa_m)^{2(K+1)} \exp\big(- s^2/2\big) \,ds \Big\}/ \sqrt{2\pi} \\ 
\nonumber & \, \leq \, {\cal O}\big\{ (2C_{X,1}/\sigma)^{2K} / (K+1)!\big\} \\ \nonumber 
& + \frac1{\{(K+1)!\}^2} E\Big\{ \kappa_m^{2(K+1)} \exp\big(\kappa_m^2\big) \big(\sqrt{2} + 2\kappa_m\big)^{2(K+1)} \Big\}\max\Big\{1,\Gamma(K+3/2)/ \sqrt{\pi}\Big\} 
 \\ \nonumber & \, =\, {\cal O}\big\{ (C_{X,1}/\sigma)^{2K} (2C_{X,1}/\sigma + \sqrt{2})^{2K} / (K+1)!\big\}\,,
\end{align} 
where we have used Assumption 2, which guarantees that $\kappa_m \leq C_{X,1}/\sigma$; the fact that $E\delta^{2(K+1)}=\Gamma(K+3/2) 2^{K+1}/\sqrt\pi$ and Minkowski's inequality. \hfill $\square$ \\

\begin{proof}[Proof of Theorem~\ref{T:1}]
Since Assumption 2 holds, we can apply Theorem~\ref{P:1}. First we consider the variance term ${\cal V}$. Using Stirling's approximation, we have
\begin{align} \nonumber & 
\exp\big(K C_{X,1}^2/\sigma^2\big) {K+m \choose K} \\ \nonumber & \, \asymp \, \frac{1}{\sqrt{2\pi}}\, \exp\big(K C_{X,1}^2 /\sigma^2\big)\, \sqrt{\frac1m + \frac1K} \cdot \big(1 + m/K\big)^K \cdot \big(1 + K/m\big)^m \\ \nonumber
& \, \leq \, \frac1{\sqrt{2\pi}} \cdot \exp\big\{K(1 + C_{X,1}^2/\sigma^2 + \log 2)\big\} \cdot (m/K)^K\,, \end{align}
since $m \geq K$ for $n$ sufficiently large. We deduce that
$$ \limsup_{n\to\infty} \sup_{P_X \in {\cal F}_X} (\log {\cal V}) / \log n \, = \, \gamma (1/\gamma-1) - 1 \, = \, -\gamma\,. $$

Using Lemma~\ref{L:smooth}, an upper bound for  $\log {\cal B}$ is given by $\mbox{const.}\cdot K - K \log K$, where the constant is uniform over all $P_X \in {\cal F}_X$, so that 
$$ \limsup_{n\to\infty} \sup_{P_X \in {\cal F}_X} (\log {\cal B}) / \log n \, = \, - \gamma\,. $$ 

Finally, under Assumption 3,  ${\cal D}={\cal O}\big(n^{-C_{X,3}C_M/2}\big)$ uniformly over all $P_X \in {\cal F}_X$. The assumption $C_M > 2/C_{X,3}$ guarantees that  ${\cal D}$ is asymptotically negligible.
\end{proof}

\subsection{Proof of Theorem~\ref{T:2}}
We define
$$ f_{\cal K}(x)  =  K^{K/2} (m-K)^{(m-K)/2} \Big\{\prod_{k \in {\cal K}} f\big(\sqrt{K} x_k\big)\Big\} \cdot \Big\{\prod_{k\in \{1,\ldots,m\}\backslash {\cal K}} \big|f\big(\sqrt{m-K} x_k\big)\big|\Big\}\,, $$
for all $x\in \mathbb{R}^m$, any integers $m > K > 0$, any subset ${\cal K} \subseteq \{1,\ldots,m\}$ with $\# {\cal K} = K$ and $f = 1_{(0,1/2]} - 1_{[-1/2,0)}$. Then we introduce the functions
$$ f_\theta(x) \, = \, {m \choose K}^{-1} \sum_{{\cal K}} \big|f_{\cal K}(x)\big| \, + \, {m \choose K}^{-1} \sum_{{\cal K}} \theta_{{\cal K}} f_{\cal K}(x)\,, $$
for any vector $\theta = \{\theta_{\cal K}\}_{\cal K}$ with $\theta_{\cal K} \in \{-1,1\}$. All $f_\theta$'s are $m$-variate Lebesgue probability densities. Then we define the probability measure $\tilde{P}_\theta$ on $\mathfrak{B}(\mathbb{R}^m)$ by 
$$ \tilde{P}_\theta(B) \, = \, (1-\eta) \cdot 1_B(0) + \eta \int_B f_\theta(x) dx\,, \qquad B\in \mathfrak{B}(\mathbb{R}^m)\,, $$
for some $\eta \in (0,1)$ still to be selected. Now let $\tilde{X}=(\tilde{X}_1,\ldots,\tilde{X}_m)$ be some $m$-dimensional random vector with the measure $\tilde{P}_\theta$. Then $P_{X,\theta}$ denotes the image measure of the functional random variable $X_1$ on $\mathfrak{B}(C_{0,0}([0,1]))$ with
\begin{equation} \label{eq:T2.1} X_1(t) \, = \, \sum_{j=1}^m \tilde{X}_j \cdot \int_0^t \varphi_j(s) ds\,, \qquad t\in [0,1]\,.  \end{equation}

Now we show that $P_{X,\theta} \in {\cal F}_X$ for all vectors $\theta$. As the $\varphi_j$'s are continuously differentiable, Assumption 1 holds true. Moreover, the support of each $f_{\cal K}$ is included in the $m$-dimensional ball around zero with the radius $1$. Therefore the measure $\tilde{P}_\theta$ is also supported on a subset of this ball so that
$ \|X_1'\|_2^2 \, = \, |\tilde{X}|^2 \, \leq \, 1\,, \qquad \mbox{a.s.}$. 
Hence Assumption 2 is satisfied. We have that 
$$ \int_0^1 \varphi_j(s) \big(\Gamma_X \varphi_{j'}\big)(s) \,ds \, = \, 1\big\{\max\{j,j'\} \leq m\big\} \cdot E \tilde{X}_j \tilde{X}_{j'} \, = \, 1\{j=j' \leq m\} \cdot \frac\eta{6m}\,, $$
for all $K\geq 3$ where we have used the fact that $f$ is an odd function. Putting
\begin{equation} \label{eq:T2.100} \eta \, = \, 6 \sqrt{m} \sqrt{C_{X,2}} \cdot \exp\big( - C_{X,3} m^\gamma /2\big)\,, \end{equation} 
for $m$ sufficiently large, Assumption 3 is satisfied as well.

Following a usual strategy for the proof of lower bounds, we bound the supremum of the statistical risk from below by the Bayesian risk where the a priori distribution of $\theta$ is such that all $\theta_{\cal K}$'s are i.i.d.~$\{-1,1\}$-valued random variables with $P(\theta_{\cal K} = 1) = 1/2$. Applying the standard formula for the minimal Bayesian risk we deduce that
\begin{align}
\sup_{P_X \in {\cal F}_X} {\cal R}\big(\hat{f}_n , f_Y\big)   \geq & E_\theta \int \big|f_{Y,\theta}(v)\big|^2 dP_V(v)  \notag\\
& -  \iint \big|E_\theta f_{Y,\theta}(u) f_{Y,\theta}^{(n)}(v)\big|^2 dP_V(u) / E_\theta f_{Y,\theta}^{(n)}(v) dP_V^{(n)}(v)\,, \label{eq:T2.10}
\end{align}
where $f_{Y,\theta}$ denotes the density of $Y_1$ with respect to $P_V$ when $X_1 \sim P_{X,\theta}$, and $P_V^{(n)}$ and $f_{Y,\theta}^{(n)}$ denote the $n$-fold product measure and product density of $P_{V}$ and $f_{Y,\theta}$, respectively. Note that $P_{Y,\theta}^{(n)}$ is the measure of the observed data. For details on the proof of (\ref{eq:T2.10}), see Section \ref{sec:side}. 

By Lemma~\ref{L:1} and Equation (\ref{eq:proj}), the $L_2(P_V)$-inner product of $f_{Y,\theta'}$ and $f_{Y,\theta''}$ equals 
\begin{align} \nonumber
\int f_{Y,\theta'}(v) &f_{Y,\theta''}(v) dP_V(v) \, = \, E f_{Y,\theta'}^{[m]}(\beta'_{V_1,1},\ldots,\beta'_{V_1,m}) f_{Y,\theta''}^{[m]}(\beta'_{V_1,1},\ldots,\beta'_{V_1,m}) \\  \nonumber
& \, = \, \int \Big\{\int g_\sigma(s - x) \, d\tilde{P}_{\theta'}(x)\Big\}\, \Big\{\int g_\sigma(s - x') d\tilde{P}_{\theta''}(x')\Big\} / g_\sigma(s) \, ds \\  \nonumber & \, = \, \iint \Big\{\int g_\sigma(s - x)\, g_\sigma(s - x') / g_\sigma(s) \, ds \Big\} \, d\tilde{P}_{\theta'}(x) \, d\tilde{P}_{\theta''}(x') \\ \nonumber
& \, = \, \iint \exp\big( x^\dagger x' / \sigma^2\big)\, d\tilde{P}_{\theta'}(x) \, d\tilde{P}_{\theta''}(x') \\ \label{eq:T2.2}
& \, =: \, \big\langle \tilde{P}_{\theta'} , \tilde{P}_{\theta''} \big\rangle_{\exp}
\,,
\end{align}
for all $\theta',\theta'' \in \{-1,1\}^{\cal K}$ since $\tilde{X}$ coincides with the vector $(\beta'_{X_1,1},\ldots,\beta'_{X_1,m})$ from (\ref{eq:eps}). Note that $\langle\cdot,\cdot\rangle_{\exp}$ represents an inner product on the linear space of all finite signed measures $Q$ on $\mathfrak{B}(\mathbb{R}^m)$ such that the support of the measure $|Q|$ is included in the $m$-dimensional closed unit ball around $0$. By a slight abuse of the notation we write $\langle f_{\cal K}, f_{\cal K'} \rangle_{\exp}$ for the corresponding inner product of the signed measures which are induced by the functions $f_{\cal K}$ and $f_{{\cal K}'}$. We show that the $f_{\cal K}$ form an orthogonal system with respect to this inner product; precisely we have that
\begin{align} \nonumber
\langle  f_{\cal K}, f_{\cal K'} \rangle_{\exp} & \, = \, \iint \exp\big( x^\dagger x' / \sigma^2\big)\, f_{\cal K}(x) \, f_{{\cal K}'}(x') \, dx\, dx' \\  \nonumber & \, = \, 1\{{\cal K} = {\cal K}'\}\cdot \Big[\iint \exp\big\{s t / (\sigma^2 K)\big\} \, f(s)\, f(t)\, ds\, dt\Big]^K \\  \nonumber & \hspace{2cm} \cdot \Big[\iint \exp\big\{s t / \big(\sigma^2 (m-K)\big)\big\} \, |f(s)|\, |f(t)|\, ds\, dt\Big]^{m-K} \\  \label{eq:T2.3} & \, = \, 1\{{\cal K} = {\cal K}'\}\cdot \big(16\, \sigma^2 K\big)^{-K} \cdot \big\{1 \pm o(1)\big\}\,,   
\end{align}
if $K$ and $m-K$ tend to infinity as $n\to\infty$. 

Combining (\ref{eq:T2.2}), (\ref{eq:T2.3}) and the fact that the $\theta_{\cal K}$'s are centered random variables we deduce that the first term in (\ref{eq:T2.10}) equals 
\begin{align}  \label{eq:T2.4}
E_\theta \int \big|f_{Y,\theta}(v)\big|^2 dP_V(v) & \, = \, \|S\|_{\exp}^2 \, + \, \eta^2 {m \choose K}^{-2} \sum_{\cal K} \big\|f_{\cal K}\big\|_{\exp}^2\,,
\end{align}
where $\|\cdot\|_{\exp}$ stands for the norm which is induced by $\langle\cdot,\cdot\rangle_{\exp}$ and the measure $S$ on $\mathfrak{B}(\mathbb{R}^m)$ is defined by 
$$ S(B) \, = \, (1-\eta) \, 1_B(0) \, + \, \eta {m \choose K}^{-1} \sum_{\cal K} \int_B |f_{\cal K}(x)| \, dx\,, \qquad B\in \mathfrak{B}(\mathbb{R}^m)\,. $$
The second term in (\ref{eq:T2.10}) is
\begin{align*}
& E_{\theta',\theta''}  \int \Big\{\int f_{Y,\theta'}(u) f_{Y,\theta''}  dP_V(u)\Big\} f_{Y,\theta'}^{(n)}(v) f_{Y,\theta''}^{(n)}(v) / \big\{E_\theta f_{Y,\theta}^{(n)}(v)\big\}  dP_V^{(n)}(v) \\
& \quad = \, \|S\|_{\exp}^2  +  \eta^2 {m \choose K}^{-2} \sum_{\cal K} \big\|f_{\cal K}\big\|_{\exp}^2 \cdot \int \big\{E_\theta \theta_{\cal K} f_{Y,\theta}^{(n)}(v)\big\}^2 / E_\theta f_{Y,\theta}^{(n)}(v)  dP_V^{(n)}(v)\,,
\end{align*}
where, here, $\theta'$ and $\theta''$ denote two independent copies of $\theta$. There we have used the fact that
\begin{align}  \label{eq:T2.11}
E_\theta &\int f_{Y,\theta}^{(n)}(v)\,  dP_V^{(n)}(v) \, = \, 1\,, \hspace{1.1cm} E_\theta \theta_{\cal K} f_{Y,\theta}^{(n)} \, = \, \frac12 E_\theta f_{Y,\theta({\cal K},+)}^{(n)} - \frac12 E_\theta f_{Y,\theta({\cal K},-)}^{(n)}\,, \end{align}
where $\theta({\cal K},\pm)$ denotes the vector $\theta$ with  $\theta_{\cal K}$ replaced by $\pm 1$; hence,
$$ \int E_\theta \theta_{\cal K} f_{Y,\theta}^{(n)}(v) \, dP_V^{(n)}(v) \, = \, 0\,. $$ 
Together with (\ref{eq:T2.4}) this implies that the right hand side of (\ref{eq:T2.10}) equals
\begin{align} \label{eq:T2.20}
\eta^2 {m \choose K}^{-2} \sum_{\cal K} \big\|f_{\cal K}\big\|_{\exp}^2 \cdot \Big[1 - \int \big\{E_\theta \theta_{\cal K} f_{Y,\theta}^{(n)}(v)\big\}^2 / E_\theta f_{Y,\theta}^{(n)}(v) \, dP_V^{(n)}(v)\Big]\,.
\end{align}
Using (\ref{eq:T2.11}) and the fact that 
$E_\theta f_{Y,\theta}^{(n)} \, = \, \frac12 E_\theta f_{Y,\theta({\cal K},+)}^{(n)} + \frac12 E_\theta f_{Y,\theta({\cal K},-)}^{(n)}$, 
we establish that
\begin{align} \nonumber
1 - \int \big\{E_\theta \theta_{\cal K}  f_{Y,\theta}^{(n)}(v)\big\}^2 / & E_\theta f_{Y,\theta}^{(n)}(v)  dP_V^{(n)}(v)  \\
\notag &\, \geq \,  2 \int \sqrt{E_\theta f_{Y,\theta({\cal K},+)}^{(n)}(v)} \sqrt{E_\theta f_{Y,\theta({\cal K},-)}^{(n)}(v)} \, dP_V^{(n)}(v)  -  1 \\ 
 \nonumber& \, \geq \, 2 E_\theta \int \sqrt{f_{Y,\theta({\cal K},+)}^{(n)}(v)} \sqrt{f_{Y,\theta({\cal K},-)}^{(n)}(v)} \, dP_V^{(n)}(v) \, - \, 1 \\
 \label{eq:T2.30}
& \, = \, 2 E_\theta \Big(\int \sqrt{f_{Y,\theta({\cal K},+)}(v)} \sqrt{f_{Y,\theta({\cal K},-)}(v)} \, dP_V(v)\Big)^n \, - \, 1\,,
\end{align}
by the Cauchy-Schwarz inequality. The Hellinger affinity between the densities $f_{Y,\theta({\cal K},+)}$ and $f_{Y,\theta({\cal K},-)}$ is bounded from below by the corresponding $\chi^2$-distance, i.e.
\begin{align*}
 \int \sqrt{f_{Y,\theta({\cal K},+)}(v)} \sqrt{f_{Y,\theta({\cal K},-)}(v)} \, dP_V(v) & \, \geq \, 1 - \frac12 \chi^2\big\{f_{Y,\theta({\cal K},+)},f_{Y,\theta({\cal K},-)}\big\}\,,
\end{align*}
where $\chi^2(f,g) = \int (f-g)^2/f \, dP_V$. We refer to the book of Tsybakov (2009) for an intensive review on these information distances. We deduce that
$$ f_{Y,\theta({\cal K},+)}(V_1) \, = \, f_{Y,\theta}^{[m]}(\beta'_{V_1,1},\ldots,\beta'_{V_1,m}) \, \geq \, 1 - \eta\,, \mbox{ a.s.}\,. $$ 
Equipped with this inequality and (\ref{eq:T2.2}) we consider that
\begin{align} \nonumber 
\chi^2\big(f_{Y,\theta({\cal K},+)},f_{Y,\theta({\cal K},-)}\big) & \, \leq \, \frac1{1-\eta}\, \big\|\tilde{P}_{\theta({\cal K},+)} - \tilde{P}_{\theta({\cal K},-)}\big\|_{\exp}^2 \, \leq \, \frac{4 \eta^2}{1-\eta}\, {m \choose K}^{-2} \big\|f_{\cal K}\big\|_{\exp}^2\,.
\end{align}
Combining this with (\ref{eq:T2.3}), (\ref{eq:T2.20}) and (\ref{eq:T2.30}) we obtain that
\begin{align} \nonumber
 \sup_{P_X \in {\cal F}_X} &{\cal R}\big(\hat{f}_n,f_Y\big) \geq  \eta^2 {m \choose K}^{-1} (16 \sigma^2 K)^{-K} \\ 
 \label{eq:T2.150}& \ \cdot \big\{1 \pm o(1)\big\} \cdot \Big(1 - \frac{2\eta^2}{1-\eta} {m \choose K}^{-2} (16 \sigma^2 K)^{-K}\cdot \big\{1 \pm o(1)\big\} \Big)^n\,.
\end{align}

Now we take $m = \lfloor (D_M \log n)^{1/\gamma}\rfloor$ and $K = \lfloor D_K (\log n) / \log(\log n) \rfloor$ for some constants $D_M,D_K > 0$. Whenever
$- D_M C_{X,3} - 2 D_K (1 / \gamma - 1) - D_K \, < \, -1$, 
the inequality (\ref{eq:T2.150}), together with (\ref{eq:T2.100}), yields that
\begin{align*}
\liminf_{n\to\infty} & \sup_{P_X \in {\cal F}_X} \big\{\log {\cal R}\big(\hat{f}_n,f_Y\big)\big\} / \log n \, \geq \, -D_M C_{X,3} - D_K /\gamma\,. 
\end{align*}
We may choose $D_K = \gamma/(2-\gamma)$ and $D_M>0$ arbitrarily close to $0$, which completes the proof of the theorem. \hfill $\square$

\subsection{Proof of Theorem~\ref{T:3}}
The proof follows a usual structure of adaptivity proofs for cross-validation techniques, see e.g. Section 2.5.1 in the book of Meister (2009) for a related proof in the field of density deconvolution.

Let $(m_n,K_n)$ be defined as in the statement of Theorem~\ref{T:1} and define the set
$$ G' \, = \, \big\{(m,K) \in G \, : \, {\cal R}(\hat{f}_Y^{[m,K]},f_Y) > 2\, {\cal R}(\hat{f}_Y^{[m_n,K_n]},f_Y)\big\}\,. $$ 
Using the notation $ \|g\|_{P_V}^2 \, = \, \int |g(x)|^2 dP_V(x) $, for any $g\in L_2(P_V)$, we need to prove that
$\lim_{n\to\infty}\sup_{P_X \in {\cal F}_X} P\big(n^{\gamma} \|\hat{f}_Y^{[\hat{m},\hat{K}]} -  f_Y\|_{P_V}^2 \geq n^{d}\big)=0$. 

By Markov's inequality we have
\begin{align} \nonumber
P&\big(n^{\gamma} \|\hat{f}_Y^{[\hat{m},\hat{K}]} -  f_Y\|_{P_V}^2 \geq n^{d}\big)  \\ \nonumber & \leq  \sum_{(m,K)\in G\backslash G'} P\big(\|\hat{f}_Y^{[m,K]} - f_Y\|_{P_V}^2  >  n^{-\gamma+d}\big)  + P\big[ (\hat{m},\hat{K}) \in G'\big] \\ \label{eq:T.3.1} &  \leq 2 (\# G)\cdot n^{\gamma-d} \cdot {\cal R}\big(\hat{f}_Y^{[m_n,K_n]},f_Y\big) + \sum_{(m,k)\in G'} P\big(\hat{m}=m,\hat{K}=K\big)\,.
\end{align}
By Theorem~\ref{T:1}, the first term in (\ref{eq:T.3.1}) converges to $0$ as $n\to\infty$ uniformly over $P_X \in {\cal F}_X$. It remains to study the second term.

On the event $\{\hat{m} = m, \hat{K} = K\}$, we have  $\mbox{CV}(m,K) \leq \mbox{CV}(m_n,K_n)$ and, hence also
\begin{align} \nonumber  
&\big\|\hat{f}_Y^{[m,K]}\big\|_{P_V}^2 - E\big\|\hat{f}_Y^{[m,K]}\big\|_{P_V}^2 - 2 \Delta_1(m,K) - 4 \Delta_2(m,K)  + {\cal R}\big(\hat{f}_Y^{[m,K]},f_Y\big) \\ \nonumber & \, \leq  \big\|\hat{f}_Y^{[m_n,K_n]}\big\|_{P_V}^2 - E\big\|\hat{f}_Y^{[m_n,K_n]}\big\|_{P_V}^2 - 2 \Delta_1(m_n,K_n) - 4 \Delta_2(m_n,K_n) \\ \label{eq:CV.4} & \hspace{8cm} + {\cal R}\big(\hat{f}_Y^{[m_n,K_n]},f_Y\big)\,,
\end{align}
where
$
\Delta_1(m,K)  =  \{n(n-1)\}^{-1} \sum_{i\neq i'} \sum_{{\bf k} \in {\cal K}(m,K)} \overline{\Xi}(i,m,{\bf k}) \cdot \overline{\Xi}(i',m,{\bf k})$, \\
$\Delta_2(m,K)  =  n^{-1} \sum_{i=1}^n \sum_{{\bf k} \in {\cal K}(m,K)} \big\{E\, \Xi(1,m,{\bf k})\big\}\cdot \overline{\Xi}(i,m,{\bf k})$, 
${\cal K}(m,K)  \, = \, \big\{ {\bf k} \in \mathbb{N}_0^m \, : \, k_1+\cdots+k_m \leq K\big\}$,
$\Xi(j,m,{\bf k})  \, = \, H_{{\bf k}}\big(\beta'_{Y_j,1}/\sigma,\ldots,\beta'_{Y_j,m}/\sigma\big)$ and \\
$\overline{\Xi}(j,m,{\bf k})  \, = \, \Xi(j,m,{\bf k}) - E\, \Xi(j,m,{\bf k})$.

The first terms of both sides of the inequality at (\ref{eq:CV.4}) can be represented as follows, using the orthonormality of the $H_{\bf k}$'s:
\begin{align*}
\|&\hat{f}_Y^{[m,K]}\|_{P_V}^2 - E  \|\hat{f}_Y^{[m,K]}\|_{P_V}^2
 \\ & \, = \, \sum_{{\bf k} \in {\cal K}(m,K)} \Big\{\Big|\frac1n \sum_{i=1}^n \Xi(i,m,{\bf k})\Big|^2 - E \Big|\frac1n \sum_{i=1}^n \Xi(i,m,{\bf k})\Big|^2\Big\} \\
& \, = \, \frac1{n^2} \sum_{i,i'} \sum_{{\bf k} \in {\cal K}(m,K)} \big\{\Xi(i,m,{\bf k}) \Xi(i',m,{\bf k}) - E\, \Xi(i,m,{\bf k}) \Xi(i',m,{\bf k})\big\} \\
& \, = \, (1-1/n) \big\{\Delta_1(m,K) + 2 \Delta_2(m,K)\big\} + \Delta_3(m,K)\,,
\end{align*}
where
$ \Delta_3(m,K) \, = n^{-2} \sum_{i=1}^n \sum_{{\bf k}\in {\cal K}(m,K)} \big\{ \Xi^2(i,m,{\bf k}) - E\, \Xi^2(i,m,{\bf k})\big\}$.
Together with \eqref{eq:CV.4} this implies that, for $(m,K) \in G'$,   
$ {\cal R}\big(\hat{f}_Y^{[m,K]},f_Y\big)/2 \, \leq \, \Delta_4(m,K,m_n,K_n)$, 
where
\begin{align*} \Delta_4&(m,K,m_n,K_n)  =   (1+1/n) \big\{\big|\Delta_1(m,K)\big| + \big|\Delta_1(m_n,K_n)\big| \\ & +  2 \big|\Delta_2(m,K) - \Delta_2(m_n,K_n)\big|\big\} + |\Delta_3(m,K)| + |\Delta_3(m_n,K_n)|\,.
\end{align*}
Hence the second term in (\ref{eq:T.3.1}) has the following upper bound:
\begin{equation} \label{eq:CV.10} 2 \sum_{(m,K)\in G'} \big\{{\cal R}\big(\hat{f}_Y^{[m,K]},f_Y\big)\big\}^{-1} \, \big\{E \Delta_4^2(m,K,m_n,K_n)\big\}^{1/2}\,. \end{equation}

In order to bound \eqref{eq:CV.10},  we need a lower bound on ${\cal R}\big(\hat{f}_Y^{[m,K]},f_Y\big)$. Theorem~\ref{P:1} provides only an upper bound to this term but an inspection of the proof of this theorem -- in particular (\ref{eq:P1.1}) to (\ref{eq:H2}) -- yields that
\begin{equation} \label{eq:CV.low}
{\cal R}\big(\hat{f}_Y^{[m,K]},f_Y\big) \, \geq \, {\cal B}(m,K) + {\cal V}^*_n(m,K) + {\cal D}^*(m) - \frac1n \big\|f_Y^{[m]}(\sigma\cdot)\big\|_{g_1}^2\,,
\end{equation}
where ${\cal B}(m,K)$ is the term ${\cal B}$ from Theorem~\ref{P:1} and
$$
{\cal D}^*(m)  \, = \, E\, \mbox{var}\{f_Y(V_1) | \mathfrak{A}_m\}\,,  \ \ 
{\cal V}^*_n(m,K)  \, = \, \frac1n {K+m \choose K}\,.
$$
Here we have used the fact that 
$ E\big\{H_{k_1,\ldots,k_m}^2(\beta'_{Y_1,1}/\sigma,\ldots,\beta'_{Y_1,m}/\sigma)\big\} \, \geq \, 1$, 
which comes from the first lines of \eqref{eq:H2}. 

In order to bound \eqref{eq:CV.10}, we also need an upper bound for $E \Delta_4^2(m,K,m_n,K_n)$, which involves $\Delta_1$ to $\Delta_3$.  For $\Delta_1$ we have
\begin{align} \nonumber
&E \big|\Delta_1(m,K)\big|^2 \, = \, \frac2{n(n-1)} \sum_{{\bf k},{\bf k}' \in {\cal K}(m,K)} \big[ \mbox{cov}\big\{\Xi(1,m,{\bf k}), \Xi(1,m,{\bf k}')\big\}\big]^2 \\ \nonumber
&  =  \frac2{n(n-1)} \sum_{{\bf k},{\bf k}' \in {\cal K}(m,K)} \big[ \big\langle H_{{\bf k}} , H_{{\bf k}'} f_Y^{[m]}(\sigma \cdot)\big\rangle_{g_1} \\ \nonumber & \hspace{6cm} - \big\langle H_{{\bf k}} , f_Y^{[m]}(\sigma \cdot)\big\rangle_{g_1} \cdot \big\langle H_{{\bf k}'} , f_Y^{[m]}(\sigma \cdot)\big\rangle_{g_1}\big]^2 \\ \nonumber &  \leq  \frac4{n(n-1)} \sum_{{\bf k},{\bf k}' \in {\cal K}(m,K)} \big\langle H_{{\bf k}} , H_{{\bf k}'} f_Y^{[m]}(\sigma \cdot)\big\rangle_{g_1}^2  \\ \nonumber & \hspace{6cm} +  \frac4{n(n-1)} \Big(\sum_{{\bf k}\in {\cal K}(m,K)} \big\langle H_{{\bf k}} , f_Y^{[m]}(\sigma \cdot)\big\rangle_{g_1}^2\Big)^2 \\ \nonumber &  \leq  \frac4{n(n-1)} \sum_{{\bf k}'\in {\cal K}(m,K)} \big\| H_{{\bf k}'} f_Y^{[m]}(\sigma \cdot) \big\|_{g_1}^2 + \frac4{n(n-1)} \big\|f_Y^{[m]}(\sigma \cdot)\big\|_{g_1}^4 \\ \label{eq:CV.6}
&  \leq  \frac4{n-1} \big\|\{f_Y^{[m]}(\sigma\cdot)\}^2\big\|_{g_1} \cdot \sup_{{\bf k}\in {\cal K}(m,K)} \big\|H_{{\bf k}}^2\big\|_{g_1}  {\cal V}^*_n(m,K)  + \frac4{n(n-1)} \big\|\{f_Y^{[m]}(\sigma\cdot)\}^2\big\|_{g_1}^2,
\end{align}
where we have used Parseval's identity with respect to the orthonormal system $H_{{\bf k}}$.

For the term involving $\Delta_2$ we have
\begin{align} \nonumber
E &\big|\Delta_2(m,K) - \Delta_2(m_n,K_n)\big|^2 \\ \nonumber & \, \leq\, \frac1n \, E\Big[\sum_{{\bf k} \in {\cal K}(m,K)} \Xi(1,m,{\bf k}) \{ E\, \Xi(1,m,{\bf k}) \} \\ \nonumber & \hspace{5cm} - \sum_{{\bf k} \in {\cal K}(m_n,K_n)} \Xi(1,m_n,{\bf k}) \{ E\, \Xi(1,m_n,{\bf k}) \}\Big]^2 \\
 \nonumber & \, = \, \frac1n \, E\Big[\sum_{{\bf k} \in \mathbb{N}_0^{\overline{m}}} \big\{1_{{\cal K}^{\overline{m}}(m,K)}({\bf k}) - 1_{{\cal K}^{\overline{m}}(m_n,K_n)}({\bf k})\big\} \Xi(1,\overline{m},{\bf k}) \{ E\, \Xi(1,\overline{m},{\bf k}) \}\Big]^2 \\
\nonumber & \, = \, \frac1n \, \sum_{{\bf k}, {\bf k}'} \big\{1_{{\cal K}^{\overline{m}}(m,K)}({\bf k}) - 1_{{\cal K}^{\overline{m}}(m_n,K_n)}({\bf k})\big\}  \big\{1_{{\cal K}^{\overline{m}}(m,K)}({\bf k}') - 1_{{\cal K}^{\overline{m}}(m_n,K_n)}({\bf k}')\big\} \\ \label{eq:CV.1000} & \hspace{3.4cm} \cdot \langle H_{\bf k} , f_Y^{[\overline{m}]}(\sigma\cdot)\rangle_{g_1} \langle H_{{\bf k}'} , f_Y^{[\overline{m}]}(\sigma\cdot)\rangle_{g_1} \big\langle H_{\bf k} , H_{{\bf k}'} f_Y^{[\overline{m}]}(\sigma\cdot)\big\rangle_{g_1}\,,
\end{align}
with $\overline{m} = \max\{m,m_n\}$ and
$ {\cal K}^{\overline{m}}(m,K) \, = \, \big\{{\bf k} \in {\cal K}(\overline{m},K) \, : \, k_l = 0, \, \forall l>m\big\}\,. $
Here we have used the fact that $H_0 \equiv 1$ and 
$ E\big\{f_Y^{[\overline{m}]}(\beta'_{V_1,1},\ldots,\beta'_{V_1,\overline{m}}) \mid \mathfrak{A}_m\big\} = f_Y^{[m]}(\beta'_{V_1,1},\ldots,\beta'_{V_1,m})$ $\mbox{ a.s.}$, 
which follows from Lemma~\ref{L:1}(a). 

Applying the Cauchy-Schwarz inequality and Parseval's identity, we get that the right hand side of (\ref{eq:CV.1000}) has the following upper bound:
\begin{align} \nonumber
 & \frac1n \Big(\sum_{{\bf k}} \big|1_{{\cal K}^{\overline{m}}(m,K)}({\bf k}) - 1_{{\cal K}^{\overline{m}}(m_n,K_n)}({\bf k})\big| \langle H_{\bf k} , f_y^{[\overline{m}]}(\sigma\cdot)\rangle_{g_1}^2\Big) \\ \nonumber & \hspace{6.5cm} \cdot \Big(\sum_{{\bf k} \in {\cal K}(\overline{m},\overline{K})} \big\|H_{{\bf k}} \cdot f_Y^{[\overline{m}]}(\sigma\cdot)\big\|_{g_1}^2\Big)^{1/2} \\
\nonumber & \, \leq \, {\cal V}^*_n(\overline{m},\overline{K}) \cdot \big\{ \big\|{\cal P}_{{\cal H}(m,K)} f_Y^{[\overline{m}]}(\sigma\cdot)\big\|_{g_1}^2 + \big\|{\cal P}_{{\cal H}(m_n,K_n)} f_Y^{[\overline{m}]}(\sigma\cdot)\big\|_{g_1}^2 \\ \nonumber & \hspace{8cm} - 2 \big\|{\cal P}_{{\cal H}(\underline{m},\underline{K})} f_Y^{[\overline{m}]}(\sigma\cdot)\big\|_{g_1}^2\big\} \\ \label{eq:CV.1001} & \hspace{1.8cm} \cdot {\overline{m} + \overline{K} \choose \overline{m}}^{-1/2} \cdot \big\|\{f_Y^{[\overline{m}]}(\sigma\cdot)\}^2\big\|_{g_1}^{1/2} \cdot \max\big\{\|H_{\bf k}^2\|_{g_1}^{1/2} : {\bf k} \in {\cal K}(\overline{m},\overline{K})\big\}\,,  \end{align}
where $\overline{K} = \max\{K,K_n\}$, $\underline{m} = \min\{m,m_n\}$, $\underline{K} = \min\{K,K_n\}$ and ${\cal P}_{{\cal H}(m,k)}$ denotes the orthogonal projector onto the linear subspace ${\cal H}(m,K)$ of $L_{2,g_1}(\mathbb{R}^m)$. 

Since
\begin{align*}
\|{\cal P}_{{\cal H}(m,K)} f_Y^{[\overline{m}]}(\sigma\cdot)\big\|_{g_1}^2 & \, = \, \|f_Y^{[m]}(\sigma\cdot)\big\|_{g_1}^2 - {\cal B}(m,K) \\ & \, = \, E |f_Y(V_1)|^2 - {\cal D}^*(m) - {\cal B}(m,K)\,, 
\end{align*} 
then using Lemma~\ref{L:1}(a), the right hand side of (\ref{eq:CV.1001}) has the following upper bound:
\begin{align} \nonumber
& 3 {\cal V}^*_n(\overline{m},\overline{K}) \cdot \big\{ {\cal D}^*(m) + {\cal D}^*(m_n) + {\cal B}(m,K) + {\cal B}(m_n,K_n)\big\}\cdot {\overline{m} + \overline{K} \choose \overline{m}}^{-1/2}\\ \label{eq:CV.1002} & \hspace{4cm} \cdot \big\|\{f_Y^{[\overline{m}]}(\sigma\cdot)\}^2\big\|_{g_1}^{1/2} \cdot \max\big\{\|H_{\bf k}^2\|_{g_1}^{1/2} : {\bf k} \in {\cal K}(\overline{m},\overline{K})\big\}\,, 
\end{align}
since ${\cal B}(m,K)$ decreases as $K$ increases.

Finally, for the term involving $\Delta_3$ we have
\begin{align} \nonumber
E\big|&\Delta_3(m,K)\big|^2 \, \leq \, n^{-3}\,  E\Big\{\sum_{{\bf k} \in {\cal K}(m,K)} \Xi^2(1,m,{\bf k})\Big\}^2 \\ \label{eq:CV.Delta3} & \, \leq \, \frac1n\, \big\{{\cal V}^*_n(m,K)\big\}^2 \, \big\|\big\{f_Y^{[m]}(\sigma\cdot)\big\}^2\big\|_{g_1} \cdot \max\big\{\|H_{\bf k}^4\|_{g_1} : {\bf k} \in {\cal K}(m,K)\big\}\,.
\end{align}

In order to bound the terms (\ref{eq:CV.6}), (\ref{eq:CV.1002}) and (\ref{eq:CV.Delta3}), we need some technical results. Using the explicit sum representation of the Hermite polynomials we write 
\begin{align} \nonumber
& \int  H_k^\ell(x) \frac1{\sqrt{2\pi}} \exp(-x^2/2) dx \\ \nonumber & =   \sum_{i_1,\ldots,i_l=0}^{\lfloor k/2 \rfloor}  \frac{(k!)^{l/2}\cdot (-2)^{-i_1-\cdots-i_\ell}}{i_1! \cdots i_\ell! \cdot (k-2i_1)! \cdots (k-2i_\ell)!} \int x^{k\ell - 2(i_1+\cdots+i_\ell)} \frac1{\sqrt{2\pi}} \exp(-x^2/2)  dx \\
\nonumber & \leq  \sum_{i=0}^{\ell \lfloor k/2 \rfloor} 2^{-i} (k!)^{\ell/2} \frac{\Gamma(k\ell/2+1/2-i)}{i! (k\ell - 2i)!} \\ \nonumber & \hspace{4.5cm} \times \sum_{i_1+\cdots+i_\ell = i} {i \choose i_1,\ldots,i_\ell} {k\ell - 2i \choose k-2i_1,\ldots,k-2i_\ell} / \sqrt{\pi} \\
\nonumber & \leq \sum_{i=0}^{\ell\lfloor k/2 \rfloor} 2^{-i} {k\ell/2 \choose i} \ell^{k\ell-i} \, \leq \, (\ell^2+\ell/2)^{k\ell/2}\,,
\end{align} 
for any $k\in \mathbb{N}_0$ and any even integer $\ell>0$. Furthermore we have
\begin{align*}
\big\|\{f_Y^{[m]}(\sigma\cdot)&\}^2\big\|_{g_1}^2  \, = \, E \big|E\big\{f_Y(V_1) | \mathfrak{A}_m\big\}\big|^4 \, \leq \, E f_Y^4(V_1) \\
& \, \leq \, E \exp\Big(-\frac2{\sigma^2} \int_0^1 |X'_1(t)|^2\, dt\Big)\, E\Big\{ \exp\Big(\frac4{\sigma^2} \int_0^1 X_1'(t)\, dV_1(t)\Big) \big| X_1\Big\} \\
& \, \leq \, \exp\big\{(8/\sigma^4)\, C_{X,1}^2\big\}\,,
\end{align*}
where we used (\ref{eq:expGauss}) and Assumption 2. Applying these results to (\ref{eq:CV.6}), (\ref{eq:CV.1002}) and (\ref{eq:CV.Delta3}) and recalling (\ref{eq:CV.low}), we deduce that (\ref{eq:CV.10}) has the upper bound
$$ (\log n)^{1+1/\gamma_0} \, D_0^{\overline{K}} \, \Big[ \big\{ n\, {\cal R}(\hat{f}_Y^{[m,K]},f_Y)\big\}^{-1/2} + {\overline{m} + \overline{K} \choose \overline{m}}^{-1/4}\Big]\,, $$  for some global finite constant $D_0>0$, so that (\ref{eq:CV.10}) converges to zero uniformly over $P_X \in {\cal F}_X$. This completes the proof of the theorem. \hfill $\square$\\[.1cm]

\section*{Acknowledgments}
{\small Delaigle's research was supported by a grant and a fellowship from the Australian Research Council. The authors are grateful to the editors and two referees for their helpful and valuable comments.}

\end{document}